\documentclass{article}
\usepackage[utf8]{inputenc}
\usepackage{amssymb,amsmath,amsthm,amscd}
\usepackage{mathrsfs}
\usepackage{nicefrac}
\usepackage{enumerate}

\usepackage{lmodern} 
\usepackage[all]{xy}

  \newcommand{\C}{\ensuremath{\mathbb{C}}}%
  \newcommand{\R}{\ensuremath{\mathbb{R}}}%
  \newcommand{\N}{\ensuremath{\mathbb{N}}}%
  \newcommand{\Z}{\ensuremath{\mathbb{Z}}}%
  \newcommand{\Q}{\ensuremath{\mathbb{Q}}}
  \newcommand{\sphere}{\ensuremath{\mathbb{S}^2}}%
\newcommand\conv{\ensuremath{\star}}

\newcommand\chap{\ensuremath{\check}}

\newcommand\CASch[2]{\ensuremath{\Lambda^{Schur}_{#1 cb}(#2)}}
\newcommand\brisure{break}
\newcommand\brisures{breaks}
\newcommand\SCHURCBAP[1]{\ensuremath{\mathrm{AP}^{Schur}_{#1 cb}}}

\def\eps{\varepsilon}

\def\R{{\mathbb R}}
\def\Z{{\mathbb Z}}
\def\N{{\mathbb N}}
\def\C{{\mathbb C}}
\def\lam{\lambda}

\def\Q{{\mathbb Q}}

\def\OO{\mathcal{O}}

\def\F{\mathbb{F}}

\def\p{\pi}

\newtheorem{thmA}{Theorem}

\newtheorem{thm}{Theorem}[section]
\newtheorem{lemma}[thm]{Lemma}
\newtheorem{corollary}[thm]{Corollary}
\newtheorem{prop}[thm]{Proposition}
\newtheorem{question}[thm]{Question}
\newtheorem{conjecture}[thm]{Conjecture}

\theoremstyle{remark}
\newtheorem{rem}[thm]{Remark}

\theoremstyle{definition}
\newtheorem{dfn}[thm]{Definition}
\newtheorem{notation}[thm]{Notation}

\begin{document}
\author{Vincent Lafforgue and Mikael de la Salle\footnote{The research of
    the second author was partially supported by ANR-06-BLAN-0015.}  }
\title{Non commutative $L^p$ spaces without the completely bounded
  approximation property}

\maketitle

\begin{abstract}
  For any $1 \leq p \leq \infty$ different from $2$, we give examples of
  non-commutative $L^p$ spaces without the completely bounded approximation
  property. Let $F$ be a non-archimedian local field. If $p>4$ or $p<4/3$
  and $r\geq 3$ these examples are the non-commutative $L^p$-spaces of the
  von Neumann algebra of lattices in $SL_r(F)$ or in $SL_r(\R)$. For other
  values of $p$ the examples are the non-commutative $L^p$-spaces of the
  von Neumann algebra of lattices in $SL_r(F)$ for $r$ large enough
  depending on $p$.
  
  We also prove that if $r \geq 3$ lattices in $SL_r(F)$ or $SL_r(\R)$ do
  not have the Approximation Property of Haagerup and Kraus. This provides
  examples of exact $C^*$-algebras without the operator space approximation
  property.
\end{abstract}

\section*{Introduction}

There are various notions of finite-dimensional approximation properties
for $C^*$-algebras and more generally operator algebras. Among others, we
can cite nuclearity, completely bounded approximation property (CBAP),
operator space approximation property (OAP), exactness... Although some of
these notions will be defined precisely in this paper, the reader is
refered to \cite{MR2391387} for an exposition of these concepts.

For the reduced $C^*$-algebra of a discrete group, most of these
approximation properties have equivalent reformulations in term of the
group~: the nuclearity of $C^*_{\mathrm{red}}(G)$ is equivalent to the amenability
of $G$. Haagerup proved in \cite{haagerup1} that the CBAP for
$C^*_{\mathrm{red}}(G)$ is equivalent to the weak amenability of $G$, and Haagerup
and Kraus \cite{MR1220905} proved that the OAP of $C^*_{\mathrm{red}}(G)$ is
equivalent to Haagerup's and Kraus' approximation property (AP) of $G$. For
equivalent formulation of exactness for a group, see \cite{MR2391387},
Chapter 5. For a discrete group, the following implications are known:
\begin{equation}\label{eq=implications_approximation} \textrm{amenability}
  \Longrightarrow \textrm{weak amenability} \Longrightarrow AP
  \Longrightarrow \textrm{exactness}. 
\end{equation}

It is also known that the first two implications are not
equivalences~: for the first one, it was proved in \cite{haagerup1}
that non-abelian free groups are weakly amenable, whereas they are not
amenable. For the second implication, a counter-example is given by
$SL_2(\Z) \ltimes \Z^2$: since AP is stable by semi-direct product
(\cite{MR1220905}), this group has the AP. But it was proved in
\cite{haagerup2} that it does not have the CBAP. In fact Haagerup
proved in \cite{haagerup2} that the reduced $C^*$-algebra of any
lattice in a locally compact simple lie group of real rank $\geq 2$
with finite center does not have the CBAP. To the knowledge of the
authors, before the present work there were no counter-example for the
implication ``$\textrm{exactness} \Longrightarrow OAP$''. But it was
conjectured by Haagerup and Kraus (\cite{MR1220905}) that the (exact)
group $SL_3(\Z)$ fails AP. We prove this conjecture (Theorem
\ref{thm=main_theorem_AP}).

Let us recall some definitions: an operator space $E$ is said to have the
completely bounded approximation property (abreviated by CBAP) if there
exists a net of finite rank linear maps $T_\alpha:E \to E$, such that
$\|T_\alpha x-x\|\to 0$ for any $x \in E$ and such that $\sup_\alpha
\|T_\alpha\|_{cb}<\infty$. The infimum over all such $T_\alpha$ of
$\sup \|T_\alpha\|_{cb}$ is the CBAP constant of $E$ and is denoted by
$\Lambda(E)$. This is the natural analogue for operator spaces of
Grothendieck's bounded approximation property (for Banach spaces). The
analogue of the metric approximation property is the completely contractive
approximation property (CCAP), and corresponds to the case when the maps
$T_\alpha$ can be taken as complete contractions. The approximation
property has also an analogue: $E$ is said to have the operator space
approximation property (OAP) if there exists a net of finite rank linear
maps $T_\alpha:E \to E$ such that for all $x \in \mathcal K(\ell^2)
\otimes_{\mathrm{min}} E$, $\|id \otimes T_\alpha(x)-x\| \to 0$. The CBAP is
stronger than OAP. As explained above these notions are of particular
interest when $E$ is an operator algebra. They are also interesting for
non-commutative $L^p$-spaces (which have a natural operator space
structure, see \cite{MR2006539}, and subsection
\ref{sec=cb_sur_LpNC}). This has been studied in \cite{MR1971296}, where
the authors discovered some nice phenomena, as a consequence of the
unpublished work from \cite{junge}~: for $1<p<\infty$, under the assumption
that the underlying von Neumann algebra is QWEP (see Remark
\ref{rem=mauvaise_extension_a_Lp}), the OAP, the CBAP and the CCAP are
equivalent properties for a non-commutative $L^p$-space.

In this paper we give examples of non-commutative $L^p$ spaces that
fail CBAP (and hence OAP by \cite{MR1971296}) for any $p \neq 2$. To
our knowledge, the only results in this direction for non-commutative
$L^p$ spaces ($p\neq 1,\infty$) were consequences of Szankowski's work
\cite{MR788259}~: he indeed proved that for $p>80$ (or $p<80/79$),
$S^p$ does not have the uniform approximation property. By an
ultrapower argument this implies the existence of non-commutative
$L^p$-spaces without the BAP (and hence without CBAP) for $p>80$ or
$p<80/79$, see Theorem 2.19 in \cite{MR2078339}.  Here we get concrete
examples for any $p\neq 2$. They are non-commutative $L^p$-spaces
associated to discrete groups, more precisely lattices in $SL_r(F)$
for $F$ a non-archimedian local field (the typical example is to take
$F$ as the field of $q$-adic numbers $\Q_q$ for some prime number $q$)
and $r$ depending on $p$, or in $SL_r(\R)$ with $r \ge3$ if $p>4$ or
$p<4/3$. More precisely, we prove the following (in the theorem below
and in the rest of the paper by a lattice in a locally compact group
$G$ we mean a discrete subgroup with finite covolume)~:

\begin{thmA}\label{thm=main_theorem}
  Let $F$ be a non-archimedian local field, $r \in \N$ with $r\geq 3$, and
  $\Gamma$ be a lattice in $SL_{r}(F)$.


If $1\leq p<\infty$ and $n \in \N^*$ are such that $r \geq 2n+1$ and
  $1 \leq p<2-2/(n+2)$ or $2+\frac{2}{n}<p < \infty$, then the
  non-commutative $L^{p}$ space of the von Neumann algebra of $\Gamma$ does
  not have the OAP or CBAP.
\end{thmA}

This theorem is proved at the end of section
\ref{sect=Contre-exemple}. Taking a direct sum of such discrete
groups, we even get a group such that the corresponding
non-commutative $L^p$ spaces do not have the CBAP for any $p \neq
2$. In the real case, we prove the following at the end of section
\ref{sect=Contre-exemple_R}~:

\begin{thmA}\label{thm=main_theorem_R}
Let $r \in \N$ with $r\geq 3$, and $\Gamma$ be a lattice in $SL_{r}(F)$
(for example $\Gamma = SL_3(\Z)$). Let $1\leq p < \infty$ with $p>4$ or
$p<4/3$.

The non-commutative $L^{p}$ space of the von
Neumann algebra of $\Gamma$ does not have the OAP or CBAP.
\end{thmA}

As a consequence of \cite{MR1971296}, the corresponding discrete
groups fail the AP. We also give an elementary proof of this. Since
linear groups are exact (\cite{MR2217050}), this gives examples of
exact groups without the AP.

\begin{thmA}\label{thm=main_theorem_AP}
Let $\Gamma=SL_3(\Z)$ or more generally a lattice in $SL_r(F)$ with
$r\geq 3$ and $F$ denoting either $\R$ or a non-archimedian local
field. $\Gamma$ does not have $AP$; equivalently the reduced
$C^*$-algebra of $\Gamma$ does not have the OAP.
\end{thmA}

To prove Theorem \ref{thm=main_theorem}, \ref{thm=main_theorem_R}, and
\ref{thm=main_theorem_AP} we introduce, for $1 \leq p \leq \infty$, a
different approximation property for a group $G$, (property
$\SCHURCBAP{p}$), in terms of completely bounded Schur
multipliers on the $p$-Schatten class on $L^2(G)$. These properties
for $p$ and $p'$ coincide if $1/p+1/p'=1$. When $p$ is $1$ or
$\infty$, this property coincides with weak amenability, and when $p$
decreases from $\infty$ to $2$, this property becomes weaker. For
discrete groups this property is implied by the completely bounded
approximation property of the corresponding non-commutative $L^p$
space, and by Haagerup's and Kraus' AP. As for the weak amenability of
a group, we introduce a constant $\CASch p G$ of the property
$\SCHURCBAP{p}$ for $G$. We notice however that for discrete
groups and $1<p<\infty$, $\CASch p G \in \{1,\infty\}$. We also prove
that the property $\SCHURCBAP{p}$ is equivalent for a locally
compact group second countable $G$ or for a lattice in $G$ (this was
proved by Haagerup in \cite{haagerup1} for the weak amenability).

The theorems above are thus consequences of the following results,
which are proved in section \ref{sect=Contre-exemple} and
\ref{sect=Contre-exemple_R} using ideas close to \cite{anniversaire}.

\begin{thmA}\label{thm=main_thm_continu}
  Let $F$ be a non-archimedian local field, $r \in \N$ with $r\geq 3$.

 If $1\leq p\leq\infty$ and $n \in \N^*$ are such that $r \geq 2n+1$
 and $1 \leq p<2-2/(n+2)$ or $2+\frac{2}{n}<p < \infty$, then
 $SL_r(F)$ does not have the property $\SCHURCBAP{p}$.
\end{thmA}

\begin{thmA}\label{thm=main_thm_continu_R}
Let $r\geq 3$. If $4<p\leq \infty$ or $1 \leq p<4/3$ then $SL_r(\R)$
does not have the property $\SCHURCBAP{p}$.
\end{thmA}
We expect that a result analogous to Theorem
\ref{thm=main_thm_continu} (with $r \to \infty$ as $p \to 2$) holds in
the real case, but this would require more work.

Let us mention that there has also been some recent activity in the study of
Herz-Schur multipliers (for $p=\infty$) for the groups $PGL_2(\Q_q)$ in
\cite{HaagSteSzw} (in relation with Schur multipliers on homogeneous trees)
and for $SL_2(\R)$ in \cite{losert}.

Let us review the organization of this paper. In a first section, we
review some basic notions on completely bounded maps between
non-commutative $L^p$-spaces, and on Schur multipliers. We give
definitions and facts on Schur multipliers on the $p$-Schatten class
on $L^2(X,\mu)$ for a general ($\sigma$-finite) measure space
$(X,\mu)$. In a digression (subsection \ref{sec=digression}), we
discuss Pisier's conjecture that there exist Schur multipliers that
are bounded on $S^p = S^p(\ell^2)$ but not completely bounded. This
conjecture is left wide open, but we reformulate it (Proposition
\ref{prop=conj_Pisier_equivalent}) and we observe that when $(X,\mu)$
has no atom, no such phenomenon can occur (Theorem
\ref{thm=norm_egal_cbnorm}), \emph{i.e.} the norm and the completely
bounded norm of a Schur multiplier coincide. Finally we prove a
characterization of Schur multipliers with continuous symbol when
$\mu$ is a Radon measure on a locally compact space~: Theorem
\ref{thm=mult_de_symb_continu}. Apart from the definitions and from
this Theorem, this section is quite independent from the rest of the
paper.

In section \ref{sect=approximation}, we introduce, for any $1 \leq p \leq
\infty$, the property of completely bounded approximation by Schur
multipliers on $S^p$ for a group and the corresponding constant $\CASch p
G$.  The main result is Theorem \ref{thm=characterzation}, which states
that the property of completely bounded approximation by Schur multipliers
on $S^p$ for a locally compact group is equivalent to the same property for
a lattice.

In section \ref{sect=lien_OAP} we restrict ourselves to discrete
groups and investigate the relationship between the property
$\SCHURCBAP{p}$ when $1<p<\infty$ and other approximation
properties (AP for the group or OAP for the non-commutative
$L^p$-space). The main results are Corollary \ref{coro=CASch1ouRien},
where we prove that $\CASch p G$ can only take the values $1$ and
$\infty$, Corollary \ref{coro=AP_implique_notrePte}, where we prove
that AP implies $\SCHURCBAP{p}$, and Corollary
\ref{coro=OAP_Lp_implieque_notrePte}, where we prove that the OAP (and
CBAP) of the associated non-commutative $L^p$-space implies
$\SCHURCBAP{p}$. The results in this section are close to
\cite{MR1971296}, but since we are working with Schatten classes $S^p$
instead of general non-commutative $L^p$-spaces, we are able to give
elementary proofs.

In section \ref{sect=Contre-exemple}, we prove Theorem
\ref{thm=main_thm_continu}. The method of the proof is similar to the
method of the proof of strong property (T) for $SL_3(F)$ in
\cite{duke}. We also derive Theorem \ref{thm=main_theorem} and the
non-archimedian case of Theorem \ref{thm=main_theorem_AP}.

In section \ref{sect=Contre-exemple_R}, we prove the same results for
$SL_{r}(\R)$ for $r \geq 3$, using again the methods close to
\cite{duke}.

\paragraph*{Acknowledgement.} We thank Gilles Pisier for numerous fruitful
discussions. We thank the anonymous referees for their comments, and
for pointing out to us the existence of \cite{losert}. We also thank
V. Losert for providing us a version of \cite{losert}.
\section{Schur multipliers on Schatten classes}
\label{sec=Schur_multi}
In this section we fix $1 \leq p \leq \infty$. Given a Hilbert space $H$,
$S^p(H)$ will denote the Schatten class on $H$: if $p=\infty$ it is the
compact operators on $H$ (equipped with the operator norm) and for
$p<\infty$ it is the set of operators $A$ on $H$ such that $\|A\|_p :=
Tr(|A|^p)^{1/p}<\infty$. This quantity is a norm which makes $S^p(H)$ a
Banach space. When $H=\ell^2_n$ then $S^p(H)$ is denoted by $S^p_n$. When
no confusion is possible we might denote $S^p(H)$ simply by $S^p$.

\subsection{CB maps on non-commutative $L^p$ spaces}
\label{sec=cb_sur_LpNC}
Note that for Hilbert spaces $H$ and $K$, the algebraic tensor product
$S^p(H) \otimes S^p(K)$ is naturally embedded in $S^p(H \otimes_2 K)$ as a
dense subspace.

A linear map $T:S^p(H) \to S^p(H)$ is called completely bounded if for any
Hilbert space $K$, the map $T^{(K)} = T \otimes id$ on $S^p(H) \otimes
S^p(K)$ extends to a bounded map on $S^p(H \otimes K)$. The completely
bounded norm of $T$ is  $\|T\|_{cb} = \sup_K \|T^{(K)}\|$. Note that
\begin{equation}\label{eq=normeCB}
\|T\|_{cb} = \|T^{(\ell^2)}\| = \sup_n \|T^{(\ell^2_n)}\|.
\end{equation}
The $n$-norm of $T$ is $\|T^{(\ell^2_n)}\|$.

This definition agrees with the definition by Pisier of the natural
operator space structure on $S^p(H)$ (\cite{MR1648908}).

More generally (if $p<\infty$) if $\mathcal M$ is a von Neumann algebra
with a semi-finite trace $\tau$, a linear map $T$ on $L^p(\mathcal M,\tau)$
is called completely bounded if $T \otimes id$ extends to a bounded
operator on $L^p(\mathcal M \bar \otimes B(H),\tau \otimes Tr)$ for any
Hilbert space $H$. Again we have that
\begin{equation}\label{eq=normeCBbis}
  \|T\|_{cb} = \sup_n \|T \otimes id: L^p(\mathcal M \otimes M_n) \to L^p(\mathcal M \otimes M_n)\|.
\end{equation}

\begin{rem}\label{rem=mauvaise_extension_a_Lp}
  When $1<p \neq 2<\infty$, it is not known whether $T$ being completely
  bounded implies that $T \otimes id$ extends to a bounded map, or
  completely bounded map, (with norm not greater than $\|T\|_{cb}$) on
  $L^p(\mathcal M\bar \otimes \mathcal N, \tau \otimes \widetilde \tau)$
  for any von Neumann algebra $\mathcal N$ with semi-finite trace
  $\widetilde \tau$. This is related to Connes' embedding problem (which is
  equivalent to the QWEP conjecture, see \cite{MR2072092} for a survey).
  When $\widetilde \tau$ if finite and $(\mathcal N,\widetilde \tau)$
  embeds in an ultraproduct of the hyperfinite $II_1$ factor, then an
  ultraproduct argument shows that the previous holds. More generally Junge
  proved \cite{junge} that this is the case when $\mathcal N$ has QWEP
  ($\mathcal N$ is said to have QWEP if $\mathcal N$ is a quotient of a
  $C^*$-aglebra with Lance's weak expectation property). For a separable
  finite von Neumann algebra, Kirchberg \cite{MR1218321} proved that QWEP
  is equivalent to the embedding into an ultraproduct of the hyperfinite
  $II_1$ factor.
\end{rem}

\subsection{Schur multipliers on $S^p(L^2(X,\mu))$.}
A Schur multiplier on $M_n(\C)$ is a linear map $T : M_n(\C) \to M_n(\C)$
of the form $T: (a_{i,j}) \mapsto (\varphi_{i,j} a_{i,j})$ for some family
$ \varphi = (\varphi_{i,j})_{1 \leq i,j \leq n}$ called the \emph{symbol}
of $T$. The multiplier $T$ is then also denoted by $M_\varphi$. We study
Schur multipliers on $S^p_n$ and their continuous generalizations.

We wish to study this notion by replacing $M_n(\C) = B(\ell^2_n)$ by more
generally $B(L^2(X,\mu))$ (or $S^p(L^2(X,\mu))$) for a $\sigma$-finite
measure space $(X,\mu)$.  Informally for a function $\varphi:X\times X \to
\C$ we are interested in the map sending an operator $T$ on $L^2(X,\mu)$
having a representation $(T_{x,y})_{x,y\in X}$ to the operator having
$(\varphi(x,y) T_{x,y})_{x,y\in X}$ as representation. But since all the
operators cannot be represented in this way we have to be more
careful. This is closely related to the notion of double operator
integrals.

For $p=2$ and any Hilbert space $H$, $S^2(H)$ is identified with the
Hilbert space tensor product $H^* \otimes_2 H$ (with the usual
identification of $\xi^* \otimes \xi$ with the rank one operator on $H$,
$\eta \mapsto \xi^*(\eta) \xi$). Let us identify (linearly isometrically)
the dual of $L^2(X,\mu)$ with $L^2(X,\mu)$ for the duality $\langle
f,g\rangle = \int f g d\mu$. We thus can identify $S^2(L^2(X,\mu))$ with
$L^2(X,\mu) \otimes_2 L^2(X,\mu) \simeq L^2(X \times X,\mu \otimes
\mu)$. We therefore have a good notion of Schur multipliers on
$S^2(L^2(X,\mu))$, which coincides with $L^\infty(X \times X,\mu \otimes
\mu)$ acting by multiplication on $L^2(X \times X,\mu \otimes
\mu)$. Thus for any $p$ and any function $\varphi \in L^\infty(X \times
X,\mu \otimes \mu)$ we say that the Schur multiplier with symbol $\varphi$
is completely bounded on $S^p$ if it maps $S^2 \cap S^p$ into $S^p$, and if
it extends to a completely bounded map from $S^p$ to $S^p$. This extension
is then necessarily unique because $S^2 \cap S^p$ is dense in $S^p$. We
denote by $M_\varphi$ this map. We will denote by
$\|\varphi\|_{MS^p(L^2(X))}$ (resp. $\|\varphi\|_{cbMS^p(L^2(X))}$) its
norm (resp. completely bounded norm).

\begin{rem} If $A$ and $B$ belong to $S^2(L^2(X,\mu))$ and correspond in
  the identification above to functions $f$ and $g$ in $L^2(X \times X,\mu
  \otimes \mu)$, then 
\begin{equation}\label{eq=lien_trace_integrale}
Tr(AB)= \int f(x,y) g(y,x) d\mu(x)d\mu(y).
\end{equation}
\end{rem}
\begin{rem}\label{rem=duality} By duality, if $1/p+1/p'=1$, the norm (resp. completely bounded
  norm) on $S^p(L^2(X))$ and $S^{p'}(L^2(X))$ of a Schur multiplier are the
  same.
\end{rem}
\begin{rem}\label{rem=monotonie_normes_Sp}
  By interpolation, this duality property implies that if $\varphi \in
  L^\infty(X \times X)$ and $2\leq p \leq q \leq \infty$, then
  $\|\varphi\|_{MS^p(L^2(X))} \leq \|\varphi\|_{MS^q(L^2(X))}$. This holds
  because $S^q(H)$ coincides isometrically with the interpolation space
  (for the complex interpolation method) $[S^p(H),S^{p'}(H)]_\theta$ for
  $1/q=\theta/{p'}+(1-\theta)/p$. In particular, for any $p$,
\[\|\varphi\|_{L^\infty(X\times X)} \leq \|\varphi\|_{MS^p(L^2(X))} \leq
\|\varphi\|_{MS^\infty(L^2(X))}.\]
The same inequalities hold for the cb-norm.
\end{rem}
The following is immediate from \eqref{eq=normeCB}.
\begin{lemma}\label{lemma=normeCBdesmult} The Schur multiplier
  corresponding to $\varphi \in L^\infty(X
  \times X,\mu \otimes \mu)$ is completely bounded on $S^p(L^2(X))$ if and
  only if the Schur multiplier corresponding to $ \widetilde \varphi
  (x,i,y,j) = \varphi(x,y)$ is bounded on $S^p(L^2(X \times \N))$ (where
  $\widetilde X=X \times \N$ is equipped with the product measure of $\mu$
  and the counting measure on $\N$). More precisely
\[\|\varphi\|_{cbMS^p(L^2(X))} = \|\widetilde
\varphi\|_{cbMS^p(L^2(\widetilde X))} = \|\widetilde
\varphi\|_{MS^p(L^2(\widetilde X))}.\] 
\end{lemma}
\begin{rem}\label{rem=produit_cartesien} 
  In fact we can replace $\N$ by any $\sigma$-finite measure space
  $(\Omega,\nu)$~: for $\varphi \in L^\infty(X \times X)$, define again
  $\widetilde X = X \times \Omega$ and $\widetilde \varphi \in
  L^\infty(\widetilde X \times \widetilde X)$ by $\widetilde \varphi
  (x,\omega,y,\omega') = \varphi(x,y)$. Then
\[\|\varphi\|_{cbMS^p(L^2(X))} = \|\widetilde
\varphi\|_{cbMS^p(L^2(\widetilde X))},\]
and this is equal to $\|\widetilde
\varphi\|_{MS^p(L^2(\widetilde X))}$ provided that $L^2(\Omega,\nu)$ is
infinite dimensional. 
\end{rem}
When $p=2$ we obviously have 
\[\|\varphi\|_{cbMS^2(L^2(X))} = \|\varphi\|_{MS^2(L^2(X))} =
\|\varphi\|_{L^\infty(X\times X)}.\] 

For $p=1,\infty$, the following characterization is well-known, and goes
back to Grothendieck (see chapter 5 of \cite{MR1818047}). The result is
more often expressed when $X=\N$, but the general statement below follows
by a martingale/ultraproduct argument. For completeness we include a proof
of this generalization, that uses Lemma \ref{lemma=martingales} below. This
proof was indicated to us by Gilles Pisier.
\begin{thm}\label{thm=caract_SchurMult} Let $(X,\mu)$ be a $\sigma$-finite
  measure space. If $p=\infty$ (or $p=1$) and $\varphi \in L^\infty(X
  \times X)$, we have that
  \[\|\varphi\|_{MS^p(L^2(X))} = \inf \|f\|_{L^\infty(X,\mu;H)}
  \|g\|_{L^\infty(X,\mu;H)}\] where the infimum runs over all separable
  Hilbert spaces, all measurable functions $f,g :X \to H$ such that
  $\varphi(x,y) = \langle f(x),g(y)\rangle$ almost everywhere.
\end{thm}

For other values of $p$, there is no known characterization of Schur
multipliers. In particular, the following conjecture of Pisier is still
open.
\begin{conjecture}[\cite{MR1648908}, Conjecture
  8.1.12]\label{conjecture=Pisier} For $1<p<\infty$, $p \neq 2$, there
  exist Schur multipliers on $S^p=S^p(\ell^2)$ that are bounded but not
  completely bounded.
\end{conjecture}
In fact there is not even an example of a Schur multiplier on $S^p_n$ (for
$n \in \N^*$ and $1<p<\infty$, $p \neq 2$) for which the norm and the
cb-norm are known to be different.

\begin{proof}[Proof of Theorem \ref{thm=caract_SchurMult}] First note that
  by Lemma \ref{lemma=change_of_measure} we can assume that $\mu$ is a
  finite measure. 

  We claim that the Theorem is equivalent to the following fact:
  \begin{equation} \label{eq=caract_Schur_S1}
\|\varphi\|_{cbMS^1(L^2(X))} = \|\varphi\|_{MS^1(L^2(X))} = \inf
  \|a\|_{L^1(\mu) \to H} \|b\|_{L^1(\mu) \to H}
\end{equation}
where the infimum runs over all Hilbert spaces $H$, all bounded linear maps
$a,b:L^1(\mu) \to H$ such that $\int \varphi(x,y) u(x) \overline{v(y)}
d\mu(x) d\mu(y) = \langle a(u),b(v)\rangle$.  

Indeed since Hilbert spaces have the Radon-Nikodym property, the Riesz
representation Theorem (\cite{MR0453964}, Chapter III) implies that a
linear map $a:L^1(\mu) \to H$ takes values in separable subspace of $H$
(hence we can assume that $H$ is separable), and $a$ is of the form
$u\mapsto \int u f d\mu$ for some map $f \in L^\infty(X,\mu;H)$ (note that
when $H$ is separable Bochner-measurable functions are simply usual
measurable functions). Then $\int \varphi(x,y) u(x) \overline{v(y)} d\mu(x)
d\mu(y) = \langle a(u),b(v)\rangle$ if and only if $\varphi(x,y) = \langle
f(x),g(y)\rangle$ almost everywhere.

Let us now prove \eqref{eq=caract_Schur_S1}. As explained before the
statement of the Theorem, we only derive the general case from the case
when $L^2(X)$ is finite dimensional. Note that the following inequalities
are easy:
\[ \|\varphi\|_{MS^1(L^2(X))} \leq \|\varphi\|_{cbMS^1(L^2(X))} \leq \inf
\|a\|_{L^1(\mu) \to H} \|b\|_{L^1(\mu) \to H}.\] The first is obvious and
the second inequality follows from Lemma \eqref{lemma=normeCBdesmult} and
from the fact that the unit ball of $S^1(L^2(X\times \N))$ is the closed
convex hull of the rank one operators in the unit ball. Let us prove the
remaining inequality. For this consider a filtration of finite
$\sigma$-subalgebras $\mathcal B_n$ such that the corresponding martingale
$\varphi_n = \mathbb E[\varphi|\mathcal B_n \otimes \mathcal B_n]$
converges almost surely to $\varphi$. For any $n$,
\eqref{eq=caract_Schur_S1} gives a Hilbert space $H_n$ and linear map $a_n,
b_n: L^1(X,\mathcal B_n,\mu) \to H_n$ such that $\int \varphi_n(x,y) u(x)
\overline{v(y)} d\mu(y) = \langle a_n(u),b_n(v)\rangle$ and such that
$\|a_n\|\|b_n\| \leq \|\varphi_n\|_{MS^1(L^2(\mathcal B_n))} +1/n$ (in fact
we can even take $\|a_n\| \|b_n\| = \|\varphi_n\|_{MS^1(L^2(\mathcal
  B_n))}$). We can and will assume that $\|a_n\|=\|b_n\|$. Take $\mathcal
U$ a non principal ultrafilter on $\N$, and let $H=\prod H_n /\mathcal U$
be the ultraproduct. It is a Hilbert space. For $u \in L^1(\mu)$ let
$u_n=\mathbb E[u|\mathcal B_n]$. If $a(u)$ (resp. $b(v)$) denotes the image
of $(a_n(u_n))_n$ (resp. $(b_n(v_n))_n$) in the ultraproduct, then $a$ and
$b$ are bounded linear maps of norm $\lim_{\mathcal U} \|a_n\|$ and
$\lim_\mathcal U \|b_n\|$. In particular by Lemma \ref{lemma=chgt_de_tribu}
$\|a\|\|b\| \leq \|\varphi\|_{MS^1(L^2(X))}$. Moreover by the dominated
convergence Theorem
\begin{eqnarray*}
\int \varphi(x,y) u(x) \overline{v(y)} d\mu(x) d\mu(y) & = & \lim_{\mathcal
  U} \int \varphi_n(x,y) u(x) \overline{v(y)} d\mu(x) d\mu(y)\\
& = & \lim_{\mathcal
  U} \int \varphi_n(x,y) u_n(x) \overline{v_n(y)} d\mu(x) d\mu(y)\\
& =& \lim_{\mathcal U} \langle a_n(u),b_n(v)\rangle = \langle a(u),b(v)\rangle.  
\end{eqnarray*}
This concludes the proof.
\end{proof}

\subsection{Change of measure.}
The first obvious remark is that for $\varphi\in L^\infty(X \times X,\mu
\otimes \mu)$, the norm (resp. cb-norm) of the corresponding Schur
multiplier on $S^p(L^2(X,\mu))$ only depends on the class of the measure
$\mu$. More precisely:
\begin{lemma} \label{lemma=change_of_measure} Let $\nu
  << \mu$ be two $\sigma$-finite measures on $X$ and $\varphi \in
  L^\infty(X\times X,\mu \otimes \mu)$. Then
  \[\|\varphi\|_{MS^p(L^2(X,\nu))} \leq
  \|\varphi\|_{MS^p(L^2(X,\mu))}.\]
The same holds for the cb-norm.
\end{lemma}
\begin{proof} If $f=d\nu/d\mu$ is the Radon-Nikodym derivative, and if $U$
  denotes the multiplication by $\sqrt f$ from $L^2(X,\nu)$ to $L^2(X,\mu)$
  ($U$ is an isometry), then $A\mapsto UAU^*$ defines a (completely)
  isometric embedding of $S^p(L^2(X,\nu))$ into $S^p(L^2(X,\nu))$ such that
  $M_\varphi(U A U^*) = U M_\varphi(A) U^*$.
\end{proof}

\subsection{Change of $\sigma$-algebra.} \label{sec=digression} We observe
basic properties of the Schur multipliers relative to conditional
expectations. Except from Lemma \ref{lemma=chgt_de_tribu} below, this
subsection is independent of the rest of the paper. We will mainly work in
the following situation:
\begin{equation}\label{eq=hyp_esp_mesures}
\begin{array}{c} \mathcal A \subset \mathcal B\textrm{ are
    $\sigma$-algebras on $X$}\\
\mu \textrm{ is a measure on $(X,\mathcal B)$ that is $\sigma$-finite on }
  (X,\mathcal A)
\end{array}
\end{equation}
Note that this allows us to talk about the conditional expectation
from $L^\infty(X,\mathcal B,\mu)$ to $L^\infty(X,\mathcal A,\mu)$
(resp. from $L^\infty(X\times X,\mathcal B \otimes \mathcal B,\mu
\otimes \mu)$ to $L^\infty(X\times X,\mathcal B \otimes \mathcal B,\mu
\otimes \mu)$). When no confusion is possible we will simply denote
$L^2(X,\mathcal B,\mu)$ by $L^2(\mathcal B)$ and $L^2(X,\mathcal
A,\mu)$ by $L^2(\mathcal A)$.

The following lemma is essentially obvious:
\begin{lemma}\label{lemma=chgt_de_tribu}
  In the situation of \eqref{eq=hyp_esp_mesures}, if $\varphi \in
  L^\infty(X \times X, \mathcal B \otimes \mathcal B,\mu \otimes \mu)$,
  \[ \|\mathbb E\left[\varphi|\mathcal A \otimes \mathcal
    A\right]\|_{MS^p(L^2(\mathcal A))} \leq \|\varphi\|_{MS^p(L^2(\mathcal
    B))} .\] The same holds for the cb-norm.
\end{lemma}
\begin{proof}
Let $V:L^2(\mathcal A) \to L^2(\mathcal B)$ be the
  isometry corresponding to the inclusion map. The map $\iota:
  B(L^2(\mathcal A)) \to B(L^2(\mathcal B))$ which maps $T$ to $VTV^*$ is a
  trace preserving $*$-homomorphism (and hence induces a complete isometry
  $S^p(L^2(\mathcal A)) \to S^p(L^2(\mathcal B))$), and the projection
  $P:B(L^2(\mathcal B))) \to B(L^2(\mathcal A))$ mapping $T$ to $V^*TV$ is
  also completely contractive on $S^p$. It remains to notice that $V^*:
  L^2(\mathcal B) \to L^2(\mathcal A)$ corresponds to the conditional
  expectation on $\mathcal A$, which implies that the following diagram
  commutes:
\[ \xymatrix{ {S^p (L^2(\mathcal B))} \ar[r]^{M_\varphi}  & {S^p (L^2(\mathcal B))}\ar[d]^{P} \\
  {S^p(L^2(\mathcal A))} \ar[u]^{\iota}\ar[r]^{M_{\mathbb E\left[\varphi|\mathcal A \otimes \mathcal
    A\right]}} & {S^p(L^2(\mathcal A))} }.\]
\end{proof}

In the vocabulary of martingales, the previous ideas become:
\begin{lemma}\label{lemma=martingales} Let $(X,\mathcal B,\mu)$ be a
  measure space and $(\mathcal B_n)_{n \in \N}$ be a filtration. Assume
  that $\mu$ is $\sigma$-finite on $(X,\mathcal B_n)$ for all $n$, and that
  $\mathcal B$ is the $\sigma$-algebra generated by $ \cup_n \mathcal
  B_n$. For any $f \in L^\infty(X \times X,\mathcal B \otimes \mathcal
  B,\mu \otimes \mu)$ let $f_n \in L^\infty(X \times X,\mathcal B_n \otimes
  \mathcal B_n,\mu \otimes \mu)$ be the conditional expectation. Then
\begin{align}
\|f\|_{MS^p(L^2(\mathcal B))} & = \lim_{n\to \infty} \nearrow
\|f_n\|_{MS^p(L^2(\mathcal B_n))}\\
\|f\|_{cbMS^p(L^2(\mathcal B))} & = \lim_{n\to \infty} \nearrow
\|f_n\|_{cbMS^p(L^2(\mathcal B_n))}
\end{align}
\end{lemma}
By $l=\lim_{n\to \infty} \nearrow u_n$ we mean that the sequence $u_n$ is
non-decreasing and converging to $l$.
\begin{rem}
  This statement remains valid replacing $(\mathcal B_n)_{n \in \N}$ by a
    filtration $(\mathcal B_\alpha)_{\alpha \in A}$ with respect to any
    directed set $A$.
\end{rem}
\begin{proof}
  The equality for the cb-norm follows from the equality for the norm and
  Lemma \ref{lemma=normeCBdesmult}. So let us focus on the inequality for
  the norm. The fact that $\|f_n\|_{MS^p(L^2(\mathcal B_n))}$ grows with
  $n$ and stays smaller than $\|f\|_{MS^p(L^2(\mathcal B))}$ is Lemma
  \ref{lemma=chgt_de_tribu}. Denote by $C$ its limit. We have to prove that
  for any $A \in S^p \cap S^2(L^2(\mathcal B))$ and $B \in S^{p'} \cap
  S^2(L^2(\mathcal B))$, we have that
  \begin{equation}\label{eq=a_prouver}
    |Tr(M_f(A) B) | \leq C \|A\|_{p} \|B\|_{p'}.\end{equation}
  But by the assumption that $\cup_n \mathcal B_n$ generate $\mathcal B$,
  $\cup_n S^p(L^2(\mathcal B_n)) $ is dense (for the norm $\|\cdot\|_p$) in
  $S^p(L^2(\mathcal B))$. We can therefore assume that $A$ (resp. $B$)
  belongs to $S^p \cap S^2 (L^2(\mathcal B_n))$ (resp. $S^{p'} \cap S^2
  (L^2(\mathcal B_n))$). But then \eqref{eq=a_prouver} follows from the
  fact $Tr(M_f(A) B) = Tr(M_{f_n}(A) B)$, which can be checked directly: let $g_A$ and $g_B \in L^2(X\times X,\mathcal B_n
  \otimes B_n,\mu \otimes \mu)$ be the functions corresponding to $A$ and
  $B$ with the identification $S^2(L^2(X)) = L^2(X \otimes X)$. Then 
  \begin{align*} Tr(M_f(A) B) &=& \int f(x,y) g_A(x,y) g_B(y,x) d\mu(x)
    d\mu(y)\\
    & =& \int f_n(x,y) g_A(x,y) g_B(y,x) d\mu(x) d\mu(y)\\
& = & Tr(M_{f_n}(A) B).
\end{align*}
\end{proof}

For the cb-norm we even have the following generalization of Remark
\ref{rem=produit_cartesien}~:
\begin{lemma}\label{lem=cnnorm_inv_par_tribu}   In the situation of \eqref{eq=hyp_esp_mesures}, if $\varphi \in
  L^\infty(X \times X, \mathcal A \otimes \mathcal A,\mu \otimes \mu)$,
\[  \|\varphi\|_{cbMS^p(L^2(X,\mathcal A))} =
\|\varphi\|_{cbMS^p(L^2(X,\mathcal B))}.\]

Therefore, if $\varphi \in L^\infty(X \times X, \mathcal B \otimes \mathcal
B,\mu \otimes \mu)$, 
\[ \|\mathbb E\left[\varphi|\mathcal A \otimes \mathcal
    A\right]\|_{cbMS^p(L^2(X,\mathcal B,\mu))} \leq
\|\varphi\|_{cbMS^p(L^2(X,\mathcal B,\mu))}.\]
\end{lemma}
\begin{proof}
  The second statement is the combination of the first statement and of
  Lemma \ref{lemma=chgt_de_tribu} for the cb-norm. So let us focus on the
  first statement. It is immediate when $\mathcal A$ is finite.

To prove the general case we can first assume that $\mu$ is a finite
measure (replacing $\mu$ by $f\mu$ for some $\mathcal A$-measurable almost
everywhere positive function $f \in L^1(X,\mathcal A,\mu)$). Then consider
a filtration $(\mathcal B_n)_{n \geq 0}$ of finite $\sigma$-subalgebras of
$\mathcal A$, such that the corresponding martingale $(\varphi_n)_{n \geq
  0}$ converges almost surely. Since $\mathcal B_n$ is finite, we get,
using that $\mathcal B_n \subset \mathcal A$ (resp. $\mathcal B_n \subset
\mathcal B$) that
\begin{equation} \label{eq=egal_surABBn} \|\varphi_n\|_{cbMS^p(L^2(X,\mathcal A))} =
\|\varphi_n\|_{cbMS^p(L^2(X,\mathcal B_n))} =
\|\varphi_n\|_{cbMS^p(L^2(X,\mathcal B))}.
\end{equation} We claim that
$\|\varphi_n\|_{cbMS^p(L^2(X,\mathcal C))} \to
\|\varphi\|_{cbMS^p(L^2(X,\mathcal C))}$ for $\mathcal C=\mathcal A$ or
$\mathcal B$. This would conclude the proof. By Lemma \ref{lemma=chgt_de_tribu}
for the cb-norm and \eqref{eq=egal_surABBn}, it is enough to prove that
$\|\varphi\|_{cbMS^p(L^2(X,\mathcal C))} \leq \limsup_n
\|\varphi_n\|_{cbMS^p(L^2(X,\mathcal C))}$. Take $A \in S^2 \cap
S^p(L^2(X\times \N,\mathcal C\otimes \mathcal P(\N))$ and $B \in S^2 \cap
S^{p'}(L^2(X\times \N,\mathcal C\otimes \mathcal P(\N))$. Let $\widetilde
X=X\times \N$, and consider
$\widetilde \varphi_n \in L^\infty(\widetilde X\times \widetilde X)$ as in
Lemma \ref{lemma=normeCBdesmult}.
Since $\widetilde \varphi_n$ converges almost surely to $\widetilde
\varphi$ and $\sup_n \|\widetilde \varphi_n\|_{L^\infty} \leq
\|\varphi\|_\infty<\infty$, the dominated convergence Theorem and
\eqref{eq=lien_trace_integrale} imply that $\lim_n Tr(M_{\widetilde
  \varphi_n}(A) B) = Tr(M_\varphi(A) B)$. Hence,
  \[ \left| Tr(M_{\widetilde \varphi}(A) B) \right| \leq \limsup_n
  \|\widetilde \varphi_n\|_{MS^p(L^2(\widetilde X,\mathcal C))} \|A\|_p
  \|B\|{p'}.\] By Lemma \ref{lemma=normeCBdesmult}, this proves the claim
  because $S^2\cap S^{p}$ (resp. $S^2\cap S^{p'}$) is dense in $S^p$
  (resp. $S^{p'}$).
\end{proof}

We do not know the answer to the following question for $1<p\neq 2<\infty$,
although we suspect that the answer should be negative~:
\begin{question}\label{question=passage_a_esp_cond}
With the same assumptions as in Lemma \ref{lemma=chgt_de_tribu}, is it true
that
\[ \|\mathbb E\left[\varphi|\mathcal A \otimes \mathcal
    A\right]\|_{MS^p(L^2(X,\mathcal B,\mu))} \leq
\|\varphi\|_{MS^p(L^2(X,\mathcal B,\mu)}?\]
\end{question}

But we can prove that this question is related to Pisier's conjecture
\ref{conjecture=Pisier}~:
\begin{prop}\label{prop=conj_Pisier_equivalent} Fix $1 \leq p \leq \infty$ and $K \geq 1$. Then the following
  are equivalent:
\begin{enumerate}[(i)]
\item \label{ass_normcbdiscret} For all $n \in \N^*$, the norm and the
  cb-norm of a Schur multiplier on $S^p_n$ are equal.
\item \label{ass_normcbcont} For all $\sigma$-finite measure space
  $(X,\mathcal B,\mu)$ and $\varphi \in L^\infty(X \times X,\mathcal B
  \otimes \mathcal B,\mu \otimes \mu)$,
\[ \|\varphi\|_{MS^p(L^2(X,\mathcal B,\mu))} =
\|\varphi\|_{cbMS^p(L^2(X,\mathcal B,\mu))}.\]
\item \label{ass_esp_condi} For all measure spaces $(X,\mathcal B,\mu)$,
  all $\sigma$-subalgebras $\mathcal A \subset B$ such that $\mu$ is
  $\sigma$-finite on $(X,\mathcal A)$, and all $\varphi \in L^\infty(X
  \times X,\mathcal B \otimes \mathcal B,\mu \otimes \mu)$,
  \[ \|\mathbb E[\varphi|\mathcal A \otimes \mathcal A]
  \|_{MS^p(L^2(X,\mathcal B,\mu))} \leq \|\varphi\|_{MS^p(L^2(X,\mathcal
    B,\mu))}.\]
\end{enumerate}
\end{prop}
\begin{rem}
  In fact the proof shows more generally that Pisier's conjecture
  \ref{conjecture=Pisier} is equivalent to the fact that there exists
  $(X,\mu)$, $\mathcal A$, $\mathcal B$ and $\varphi$ as in
  (\ref{ass_esp_condi}) such that $\|\varphi\|_{MS^p(L^2(X,\mathcal
    B,\mu))}<\infty$ but the Schur multiplier with symbol
  $\mathbb E\left[\varphi|\mathcal A \otimes \mathcal
    A\right]$ is not bounded on $MS^p(L^2(X,\mathcal
    B,\mu))$.
\end{rem}
\begin{proof}
  First remark that since any $\sigma$-finite measure is equivalent to a
  probability measure, both assertions (\ref{ass_normcbcont}) and
  (\ref{ass_esp_condi}) are equivalent to the same assertions with $\mu$
  being a probability measure.

  The assertion (\ref{ass_normcbdiscret}) is just (\ref{ass_normcbcont})
  restricted to the case when $\mathcal B$ is finite. Thus
  (\ref{ass_normcbcont}) implies (\ref{ass_normcbdiscret}) and the other
  direction follows by Lemma \ref{lemma=martingales} (or rather the remark
  following, applied to the filtration of all finite $\sigma$-subalgebras
  of $\mathcal B$, provided that $\mu$ is finite).

  $(\ref{ass_normcbcont})\Rightarrow(\ref{ass_esp_condi})$ follows from
  Lemma \ref{lem=cnnorm_inv_par_tribu}.

  Let us prove now that
  $(\ref{ass_esp_condi})\Rightarrow(\ref{ass_normcbcont})$. Let
  $(X,\mathcal B,\mu)$ be a $\sigma$-finite measure space and $\varphi \in
  L^\infty(X \times X)$. Let $\widetilde X = X \times \N$ and define
  $\widetilde \varphi$ on $\widetilde X \times \widetilde X$ by
\[\widetilde \varphi(x,i,y,j)=\left\{\begin{array}{ll}\varphi(x,y) & \textrm{if
    }i=j=0\\ 0 & \textrm{otherwise.}\end{array}\right.\]
Fix $\eps>0$ and consider the probability measure $P_\eps$ on $\N$ such that
$P_\eps(0) = 1-\eps$ and $P_\eps(i) = \eps 2^{-i}$ if $i>0$. Let $\mathcal
B_1= \mathcal B \otimes \mathcal P(\N)$ and $\mathcal A_1 = \mathcal B
\otimes \{ \emptyset; \N\}$. Then the conditional expectation of
$\widetilde \varphi$ with respect to $P_\eps$ on $\mathcal A_1 \otimes \mathcal A_1$ is $\mathbb E[\widetilde \varphi |
\mathcal A_1 \otimes \mathcal A_1](x,i,y,j)=(1-\eps)\varphi(x,y)$. But the equality
\[\|\widetilde \varphi\|_{MS^p(L^2(\widetilde X,\mathcal B_1,\mu \otimes P_\eps))} =
\|\varphi\|_{MS^p(L^2(X,\mu))}\]
is obvious, whereas the equality
\[\|\mathbb E[\widetilde \varphi | \mathcal A_1 \otimes \mathcal
A_1]\|_{MS^p(L^2(\widetilde X, \mathcal B_1 \mu \otimes P_\eps))} =
(1-\eps) \|\varphi\|_{cbMS^p(L^2(X,\mathcal B,\mu))}\] follows from the
fact that $P_\eps$ is equivalent to the counting measure on $\N$ and from
Lemma \ref{lemma=normeCBdesmult}. The assumption (\ref{ass_esp_condi}) thus 
implies that
\[ (1-\eps)\|\varphi\|_{cbMS^p(L^2(X,\mathcal B,\mu))} \leq
\|\varphi\|_{cbMS^p(L^2(X,\mathcal B,\mu))}.\]
Making $\eps \to 0$ we get (\ref{ass_normcbcont}).
\end{proof}

The following Lemma gives a positive answer to question
\ref{question=passage_a_esp_cond}, in the setting when ``the conditional
expectation is implemented by random permutations''. By an atom in a
measure space $(X,\mathcal B,\mu)$, we mean a measurable subset that cannot
be partitioned into two subsets of positive measure.
\begin{lemma}\label{lem=random_permu} Let $\mathcal A \subset \mathcal B$
  be two \emph{finite} $\sigma$-algebras on $X$, $\mu$ a finite measure on
  $(X,\mathcal B)$ such that every atom of $\mathcal A$ is partitioned into
  atoms of $\mathcal B$ \emph{of same measure}. Then for any $\mathcal B
  \otimes \mathcal B$-measurable $\varphi:X \times X \to \C$, 
  \[ \|\mathbb E\left[\varphi|\mathcal A \otimes \mathcal
    A\right]\|_{MS^p(L^2(X,\mathcal B))} \leq
  \|\varphi\|_{MS^p(L^2(X,\mathcal B))}.\]
\end{lemma}
\begin{proof}
  We can as well assume that $X$ is a finite set and $\mathcal B=\mathcal
  P(X)$. If $\sigma$ and $\sigma'$ are permutations of $X$, denote by
  $\varphi^{\sigma,\sigma'} (x,y) =\varphi(\sigma(x),\sigma'(y))$. Note that by
  invariance of the norm on $S^p(\ell^2(X))$ by permutation of rows and
  columns
  \begin{equation} \label{eq=mult_permutation}\|\varphi\|_{MS^p(L^2(X,\mathcal B))} =
    \|\varphi^{\sigma,\sigma'}\|_{MS^p(L^2(X,\mathcal B))}.\end{equation}

  Let now $\sigma$ be a random permutation of $X$ satisfying the following:
  for any atom $A$ of $\mathcal A$, $\sigma(A)=A$ and for any $x,y \in A$,
  the probability that $\sigma(x)=y$ is $1/|A|$. Let $\sigma'$ be an
  independent copy of $\sigma$. Then for any $x,y \in X$, $\mathbb
  E\left[\varphi|\mathcal A \otimes \mathcal A\right](x,y)$ is the expected
  value of $\varphi^{\sigma,\sigma'}(x,y)$, and the triangle inequality and
  \eqref{eq=mult_permutation} conclude the proof.
\end{proof}

We can thus conclude by the following result:
\begin{thm}\label{thm=norm_egal_cbnorm}
  Let $(X,\mathcal B,\mu)$ be a $\sigma$-finite measure space with no atom.
  Then for any $\varphi \in L^\infty(X \times X,\mathcal B \otimes \mathcal
  B,\mu \otimes \mu)$ and any $1 \leq p \leq \infty$,
  \[ \|\varphi\|_{MS^p(L^2(X,\mathcal B))} = \|\varphi\|_{cbMS^p(L^2(X,\mathcal B))} .\]
\end{thm}
\begin{proof}
Replacing $\mu$ by a probability measure which is equivalent, we can assume
that $\mu$ is a probability measure.

By Lemma \ref{lemma=martingales} it is enough to prove that for any finite
$\sigma$-subalgebra $\mathcal A\subset \mathcal B$, if $\varphi_{\mathcal
  A} = \mathbb E[\varphi|\mathcal A \otimes \mathcal A]$, then
\[ \|\mathbb E\left[\varphi|\mathcal A \otimes \mathcal
    A\right]\|_{cbMS^p(L^2(\mathcal A))}
\leq \|\varphi\|_{MS^p(L^2(\mathcal B))}.\] Fix such $\mathcal A$, and some
integer $n$. Use the assumption that $\mathcal B$ has no atom: every atom
$A$ of $\mathcal A$ can be partitioned into $n$ $\mathcal B$-measurable
subsets $A^{1},\dots,A^{n}$ of same measure $\mu(A)/n$. Let $\mathcal B'$
be the $\sigma$-algebra generated by the set $A^{i}$ for $1\leq i \leq n$
and $A$ atom of $\mathcal A$. Then by Lemma \ref{lem=random_permu},
\[ \|\mathbb E\left[\varphi|\mathcal A \otimes \mathcal
    A\right]\|_{MS^p(L^2(\mathcal B'))} \leq
\|\mathbb E\left[\varphi|\mathcal B' \otimes \mathcal
    B'\right]\|_{MS^p(L^2(\mathcal B'))}.\] But the left-hand
side is equal to the norm of $\mathbb E\left[\varphi|\mathcal A \otimes \mathcal
    A\right] \otimes id$ acting on
$S^p(L^2(\mathcal A) \otimes \ell^2_n)$, and the right-hand side is by
Lemma \ref{lemma=chgt_de_tribu} not greater than
$\|\varphi\|_{MS^p(L^2(X,\mathcal B))}$. Since $n$ was arbitrary, this
concludes the proof.
\end{proof}

\subsection{Multipliers with continuous symbol.} We now study Schur
multipliers in the setting when $X$ is a locally compact space, $\mu$ is a
$\sigma$-finite Radon measure, and the symbol $\varphi$ is continuous.

\begin{thm}\label{thm=mult_de_symb_continu}
  Let $\mu$ be a $\sigma$-finite Radon measure on a locally compact space
  $X$, and $\varphi:X \times X \to \C$ a continuous function. Let $1 \leq p
  \leq \infty$ and $C>0$. The following are equivalent:
\begin{enumerate}[(i)] 
\item \label{asser=norm_locCompact} $\varphi$ defines a bounded multiplier
  on $S^p(L^2(X,\mu))$ with norm less than $C$.
\item \label{asser=norm_parties_finies} For any finite subset
  $F=\{x_1,\dots,x_N\}$ in $X$ belonging to the support of $\mu$, the
  multiplier $(\varphi(x_i,x_j))$ is bounded on $S^p(\ell^2(F))$ with norm
  less than $C$.
\end{enumerate}
The same equivalence is true for the cb-norms.

In particular, the norm and cb-norm on $S^p$ of the multiplier with symbol
$\varphi$ only depends on the support of $\mu$, and if this support has no
isolated point, its norm and cb-norm coincide.
\end{thm}
\begin{proof}
  Since any $\sigma$-finite Radon measure is equivalent to a finite
  measure, we can assume that $\mu$ is a probability measure.

  Let us first prove that
  (\ref{asser=norm_locCompact})$\Rightarrow$(\ref{asser=norm_parties_finies}).
  Assume (\ref{asser=norm_locCompact}) and fix a finite subset
  $F=\{x_1,\dots,x_N\}$ of the support of $\mu$. Then for any family
  $V_1,\dots,V_N$ of disjoint Borel subsets such that $x_i \in V_i$ and
  $\mu(V_i)>0$, we can consider $\mathcal A$ the $\sigma$-subalgebra of
  $\mathcal B$ generated by the $V_i$'s. By Lemma
  \ref{lemma=chgt_de_tribu}, we get that the norm on $S^p_n$ of the Schur
  multiplier with symbol given by
  \[ (i,j) \mapsto \textrm{average value of $\varphi$ on }V_i \times V_j\]
  is not greater than the norm on $S^p(L^2(X))$ of $M_\varphi$, \emph{i.e.}
  is not greater than $C$. But if the $V_i$'s are chosen to be contained in
  arbitrary small neighbourhouds of $x_i$ (which is possible because $x_i$
  belongs to the support of $\mu$), we get at the limit that the average
  value of $\varphi$ on $V_i \times V_j$ tends to $\varphi(x_i,x_j)$. This
  proves (\ref{asser=norm_parties_finies}).

  For the converse, assume (\ref{asser=norm_parties_finies}). By a density
  argument it is enough to prove that
  \[|Tr(M_\varphi(A)B)|\leq C \|A\|_p \|B\|_{p'}\] for finite rank
  operators on $A$ and $B$ on $L^2(X,\mu)$ that correspond to elements
  $g_A,g_B$ of $C_c(X) \otimes C_c(X)$ in the identification
  $S^2(L^2(X,\mu)) = L^2(X \times X,\mu \otimes \mu)$ (here $C_c(X)$
  denotes the continuous functions from $X$ to $\C$ with compact
  support). Find $(\mu_\alpha)$ a net a probability measures on $X$ with
  finite support contained in the support of $\mu$ converging vaguely to
  $\mu$ (\emph{i.e.} such that $\int f d\mu_\alpha \to \int f d\mu$ for all
  $f \in C_c(X)$). For the existence of such a net, see \cite{BourbakiInt},
  Chap. IV, \S 2, 4, Corollaire 2. Then for any $\alpha$ denote by
  $A_\alpha$ and $B_\alpha$ the operators on $L^2(X,\mu_\alpha)$
  corresponding to $g_A$ and $g_B$ viewed in $L^2(X,\mu_\alpha) \otimes
  L^2(X,\mu_\alpha)$. We claim that $\lim_\alpha \|A_\alpha\|_p = \|A\|_p$
  and $\lim_\alpha \|B_\alpha\|_{p'} = \|B\|_{p'}$. This would conclude the
  proof of
  (\ref{asser=norm_parties_finies})$\Rightarrow$(\ref{asser=norm_locCompact})
  since by \eqref{eq=lien_trace_integrale} and the vague convergence of
  $\mu_\alpha$ to $\mu$, we have that
\[Tr(M_f(A_\alpha) B_\alpha) \xrightarrow{\alpha} Tr(M_f(A)B).\]

To prove the claim (say for $A$), write (using the Gram-Schmidt
orthonormalization process) $g_A = \sum_{i,j=1}^N a_{i,j} f_i \otimes f_j$
for a family $f_i\in C_c(X)$ which is orthonormal in $L^2(\mu)$. Thus
$\|A\|_p = \|(a_{i,j})_{i,j \leq N}\|_{S^p_N}$. But by the vague
convergence of $\mu_\alpha$ to $\mu$, the family $f_1,\dots,f_N$ is almost
orthonormal in $L^2(X,\mu_\alpha)$, and thus it is close to an orthonormal
family $f^\alpha_i$, and thus we can write $g_A = \sum_{i,j=1}^N
a_{i,j}^\alpha f^\alpha_i \otimes f^\alpha_j$ with $a_{i,j}^\alpha$
converging to $a_{i,j}$. This indeed implies that
\[ \|A_\alpha\|_p = \|(a_{i,j}^\alpha)_{i,j \leq N}\|_{S^p_N}
\xrightarrow{\alpha} \|(a_{i,j})_{i,j \leq N}\|_{S^p_N} = \|A\|_p.\]

This proves
(\ref{asser=norm_locCompact})$\Leftrightarrow$(\ref{asser=norm_parties_finies}).
For the cb-norm, apply this equivalence with $X$ replaced by $X \times \N$
and use Lemma \ref{lemma=normeCBdesmult}.

It remains to note that when the support of $\mu$ has no isolated point,
the norm and cb-norm of a Schur multiplier coincide. We show that the best
$C$ such that (\ref{asser=norm_parties_finies}) holds is equal to the best
$C$ such that (\ref{asser=norm_parties_finies}) holds for the cb-norm. For
this, fix a finite subset $F=\{x_1,\dots,x_N\}$ in the support of $\mu$ and
an integer $n$. For any $1\leq i \leq N$, find $n$ nets
$(y_\alpha^{i,j})_{\alpha}$ for $j=1,\dots,n$ of elements of the support of
$\mu$ such that $y_\alpha^{i,j} \xrightarrow{\alpha} x_i$ 
and such that for fixed $\alpha$, the $y_\alpha^{i,j}$ for $1\leq i \leq N$
and $j=1,\dots,n$ are all disctinct. This is possible because the support
of $\mu$ has no isolated point. Note that
$\varphi(y_\alpha^{i,j},y_\alpha^{i',j'}) \xrightarrow{\alpha}
\varphi(x_i,x_{i'})$. Expressing, for any $\alpha$,
(\ref{asser=norm_parties_finies}) with the finite set
$\{y_\alpha^{i,j},1\leq i \leq N, 1 \leq j \leq n\}$, one gets at the limit
that the $n$-norm of the multiplier with symbol $\varphi(x_i,x_{i'})$ is
bounded by $C$. But $n$ was arbitrary.
\end{proof}

\section{Approximation by Schur multipliers}
\label{sect=approximation}
In this section locally compact groups will always be assumed to be second
countable. The reason is that we want to deal with $\sigma$-finite measure
spaces, and a Haar measure on a locally compact group is $\sigma$-finite if
and only if the group is second countable.

Recall that the Fourier algebra $A(G)$ of a locally compact group $G$ is
the set of coefficients of the left regular representation of $G$ and is
naturally identified with the predual of the von Neumann algebra of $G$.

\begin{notation}
For a locally compact and second countable group $G$ (say equipped with a
left Haar measure) and a function $\varphi \in L^\infty(G)$ we will denote
by $\chap \varphi \in L^\infty(G\times G)$ the function defined by $\chap
\varphi(g,h)=\varphi(g^{-1}h)$. The corresponding Schur multiplier is
sometimes called Toeplitz-Schur multiplier, or Herz-Schur
multiplier.
\end{notation}
Bo{\.z}ejko-Fendler's characterization \cite{MR753889} (see also
\cite{BozeFend}) states that for 
$\varphi:G \to \C$, the completely bounded norm on $VN(G)$ of the Fourier
multiplier $\lambda(g) \mapsto \varphi(g) \lambda(g)$, denoted by
$\|\varphi\|_{M_0A(G)}$ (by duality it is the cb-norm of the multiplication
by $\varphi$ on $A(G)$) is equal to the norm of the Schur multiplier $\chap
\varphi$~:
\[ \| \chap \varphi\|_{cbMS^\infty(L^2(\Gamma))} = \| \varphi(g)\|_{M_0A(G)}.\]

As defined in \cite{haagerup3}, $G$ is said to be weakly amenable if there
exists a constant $C$ and a net $\varphi_\alpha \in A(G)$ that converges
uniformly on compact subsets to $1$ and such that $\|
\varphi_\alpha(g)\|_{M_0A(G)}\leq C$. The infimum of such $C$ is denoted by
$\Lambda_G$.

We generalize this notion as follows~:
\begin{dfn}
  If $G$ is a locally compact second countable group and $1\leq p \leq
  \infty$, we say that $G$ has the \emph{property of completely bounded
    approximation by Schur multipliers on $S^p$} ($\SCHURCBAP{p}$) if there
  is a constant $C$, a net of functions $\varphi_\alpha \in A(G)$ such that
  $\varphi_\alpha \to 1$ uniformly on compact subsets of $G$ and such that
  $\|\chap \varphi_\alpha\|_{cbMS^p(L^2(G))} \leq C$. The infimum of such
  $C$ is denoted by $\CASch{p}{G}$.
\end{dfn}

Note that if $G$ is not discrete, Theorem \ref{thm=mult_de_symb_continu}
shows that the condition $\|\chap \varphi_\alpha\|_{cbMS^p(L^2(G))} \leq C$
is equivalent to $\|\chap \varphi_\alpha\|_{MS^p(L^2(G))} \leq C$.

Here are some basic properties of $\CASch{p}{G}$:
\begin{prop}\label{thm=basic_propertiesofCASch} For a locally compact
  second countable group $G$:
\begin{itemize}
\item For $p=\infty$, $G$ has the property of completely bounded
approximation by Schur multipliers on $S^p$ if and only if it is weakly
amenable, and
\[\CASch{\infty}{G} = \Lambda_{G}.\]
\item $\CASch{2}{G}=1$.
\item If $1 \leq p \leq \infty$, and $1/p + 1/p'=1$ then $\CASch{p}{G} =
\CASch{p'}{G}$.
\item If $2 \leq p \leq q \leq \infty$, then $\CASch{p}{G} \leq
\CASch{q}{G}$.
\item If $H$ is a closed subgroup of $G$ and $1 \leq p \leq \infty$,
  $\CASch{p}{H} \leq \CASch{p}{G}$.
\end{itemize}
\end{prop} 
\begin{proof}
  The first point is by definition of weak amenability and of $\Lambda_G$.

  The second assertion is obvious because for any $\varphi \in L^\infty(G)$,
  $\|\chap \varphi\|_{cbMS^2(L^2(G))} = \|\varphi\|_\infty$. The next two
  assertions are consequences of Remarks \ref{rem=duality} and
  \ref{rem=monotonie_normes_Sp}. The last assertion is a consequence of
  Theorem \ref{thm=mult_de_symb_continu} (remember that $A(G) \subset C(G)$).
\end{proof}

It is also natural to study the approximation by continuous functions with
compact support. This yields to a property which might be weaker in general
but which is equivalent when the group is discrete (by the proof of Theorem
\ref{thm=characterzation}, we also get the same notion when $G$ contains a
lattice).
\begin{lemma}\label{lemma=def_equivalente_par_AG} Let $G$ be a locally
  compact second countable group, and $1 \leq p \leq \infty$. In the
  definition of $\CASch{p}{G}$ the functions $\varphi_\alpha$ can be taken
  in $A(G) \cap C_c(G)$.

In particular when $G$ is discrete $\CASch{p}{G}$ is the smallest $C$ such
  that there exists a net of functions with finite support
  $\varphi_\alpha:G \to \C$ such that $\varphi_\alpha(g) \to 1$ for all $g
  \in G$ and such that $\|\chap \varphi_\alpha\|_{cbMS^p(L^2(G))} \leq C$.
\end{lemma}
\begin{proof}
  The first point is because $C_c(G)$ is dense in $A(G)$ and because, by
  Remark \ref{rem=monotonie_normes_Sp} and the inequality
  $\|\cdot\|_{M_0A(G)} \leq \|\cdot\|_{A(G)}$, for any $\varphi \in A(G)$,
  \[ \|\varphi\|_\infty\leq \|\chap \varphi\|_{cbMS^p(L^2(G))} \leq \|\varphi\|_{M_0A(G)} \leq \|\varphi\|_{A(G)}.\]

The second statement is because $A(G)$ contains all functions with finite
support when $G$ is discrete.
\end{proof}

\subsection{From a lattice to the whole group}
In this subsection we prove that the property of completely bounded
approximation by Schur multipliers on $S^p$ for a group is equivalent
to the same property for a lattice in this group. This was proved in
\cite{haagerup2} for $p=\infty$. With the tools developped in section
\ref{sec=Schur_multi}, the proof is the very close to Haagerup's
proof. The main result is~:
\begin{thm}\label{thm=characterzation} Let $G$ be a locally compact second
  countable group and $\Gamma$ a lattice in $G$. Then for $1\leq p
  \leq \infty$
\[\CASch{p}{G} = \CASch{p}{\Gamma}.\]
\end{thm}

We now fix $p$, and $G$, $\Gamma$ as in Theorem \ref{thm=characterzation}.
We denote by $\mu$ a Haar measure on $G$ ($\mu$ is a left and right Haar
measure because a group containing a lattice is unimodular). Let $\Omega$
be a Borel fundamental domain of the action of $\Gamma$ by
right-multiplication on $G$, \emph{i.e.} $\Omega$ is a Borel subset of $G$
such that the restriction of the quotient map $G \to G/\Gamma$ is
bijective. Since $\Gamma$ is a lattice, $\Omega$ has finite Haar measure,
and we can assume that it has measure $1$. For $g \in G$ denote by
$g=\omega(g) \gamma(g)$ the unique decomposition of $g$ with $\omega(g) \in
\Omega$ and $\gamma(g) \in \Gamma$.

For any bounded function $\psi:\Gamma \to \C$ we define $\varphi:G \to
\C$ by
\[\varphi = \chi_\Omega \conv \psi \mu_\Gamma \conv
\widetilde \chi_\Omega,\] where $\mu_\Gamma$ is the counting measure on
$\Gamma$, and $\chi_\Omega$ (resp. $\widetilde \chi_\Omega$) is the
characteristic function of $\Omega$ (resp. $\Omega^{-1}$). Equivalently,
\[\varphi(g) = \int_\Omega \psi(\gamma(g\omega)) d\mu(\omega).\]
\begin{lemma}\label{lemma=discrete_to_cont}
For $\psi:\Gamma \to \C$, 
\[\| \chap \varphi\|_{cbMS^p(L^2(G))} \leq \| \chap
\psi\|_{cbMS^p(L^2(\Gamma))}.\]
\end{lemma}
\begin{proof}
Since for any $h \in G$, the measure $\mu{\left|_\Omega\right.}$ is
invariant under $ \omega' \mapsto \omega( h \omega')$, and since $ g
h^{-1} \omega(h\omega') = \omega(g\omega') \gamma(g \omega')
\gamma(h\omega')^{-1}$, we get that
\[\varphi(g h^{-1}) = \int_\Omega \psi(\gamma(g\omega')
\gamma(h\omega')^{-1}) d\mu(\omega').\] 
By Fubini's theorem it is enough to prove that for any $\omega' \in \Omega$
the Schur multiplier with symbol $(g,h) \mapsto \psi(\gamma(g\omega')
\gamma(h\omega')^{-1})$ has cb-norm on $S^p(L^2(G))$ not larger than the
cb-norm on $S^p(\ell^2(\Gamma))$ of the Schur multiplier with symbol
$(\gamma,\gamma') \mapsto \psi(\gamma \gamma'^{-1})$. But since
measure-theoretically, we have $G=\Gamma \times \Omega$ for the
identification of $g$ with $(\gamma(g\omega'),\omega(g\omega'))$
these Schur multipliers have in fact the same cb-norm, by Remark
\ref{rem=produit_cartesien}.
\end{proof}

We will also use the following Lemma from \cite{haagerup2}. Since
\cite{haagerup2} is not easily available we reproduce a proof.
\begin{lemma}\label{lem=dansAGammaDansAG}
$\|\varphi\|_{A(G)} \leq \|\psi\|_{A(\Gamma)}$.
\end{lemma}
\begin{proof}
  If $\psi \in A(\Gamma)$ there exist $f,g \in \ell^2(\Gamma)$ such that
  $\|f\|_2 \|g\|_2 = \|\psi\|_{A(\Gamma)}$ and $\varphi = f \conv
  \widetilde g$ where $\widetilde g(\gamma) = g(\gamma^{-1})$. Put $f_1=f
  \mu_\Gamma \conv \chi_\Omega$ and $g_1=g \mu_\Gamma \conv
  \chi_\Omega$. Then $f_1 \conv \widetilde g_1= \varphi$ and hence
  \[\|\psi\|_{A(G)} \leq \|f_1\|_{L^2(G)} \|g_1\|_{L^2(G)} =
  \|f\|_{\ell^2(\Gamma)} \|g\|_{\ell^2(\Gamma)} = \|\psi\|_{A(\Gamma)}.\]
\end{proof}

\begin{proof}[Proof of Theorem \ref{thm=characterzation}]
  The inequality $\CASch{p}{G} \geq \CASch{p}{\Gamma}$ holds for any
  closed subgroup of $G$ by Proposition \ref{thm=basic_propertiesofCASch}.

  For the other inequality, let $\psi_\alpha \in A(\Gamma)$ converging
  pointwise to $1$ and such that $\sup_\alpha \|\chap
  \psi_\alpha\|_{cbMS^p(\ell^2(\Gamma))} = \CASch{p}{\Gamma}$. Use Lemma
  \ref{lemma=discrete_to_cont} and define $\varphi_\alpha^0 = \chi_\Omega
  \conv \psi_\alpha \mu_\Gamma \conv \widetilde \chi_\Omega$. Lemma
  \ref{lemma=discrete_to_cont} implies that $\| \chap
  \varphi_\alpha^0\|_{cbMS^p(L^2(G))} \leq \| \chap
  \psi_\alpha\|_{cbMS^p(\ell^2(\Gamma))}$. Also, by Lemma
  \ref{lem=dansAGammaDansAG}, $\varphi_\alpha^0 \in A(G)$.  However
  $\varphi_\alpha^0$ may not converge to $1$ uniformly on compact subsets.
  Take $h \in C_c(G)^+$ (a continuous nonnegative function with compact
  support) such that $\int h d\mu=1$, and define $\varphi_\alpha = h \conv
  \varphi_\alpha^0$. Then $\chap \varphi_\alpha$ is the average with
  respect to the probability measure $h(x) d\mu(x)$ of $(s,t)\mapsto
  \chap\varphi_\alpha(sx,t)$. But for any $x$, the Schur multiplier with
  symbol $(s,t)\mapsto \chap\varphi_\alpha(sx,t)$ has same norm as the
  multiplier with symbol $\chap \varphi_\alpha$. This implies that $\|
  \chap \varphi_\alpha\|_{cbMS^p(L^2(G))} \leq \| \chap
  \psi_\alpha\|_{cbMS^p(\ell^2(\Gamma))}$. In the same way, since left
  translations by $G$ act on $A(G)$ isometrically, $\varphi_\alpha \in
  A(G)$.  The fact that $\lim_\alpha \varphi_\alpha(g)=1$ follows from the
  dominated convergence Theorem in
\begin{equation*}
\varphi_\alpha(g) = \int_G \int_\Omega h(gs^{-1}) \psi_\alpha(
  \gamma(s\omega')) d\mu(\omega') d\mu(s)
\end{equation*}
 The convergence is uniform in compact subsets of $G$ because the family
$h(g\cdot)$, when $g$ belong to a compact subset of $G$, is relatively
compact in $L^1(G)$.
\end{proof}

\section{The case of discrete groups}
\label{sect=lien_OAP}
In this section we restrict ourselves to discrete groups and we study
the relation between the property of completely bounded approximation
by Schur multipliers on $S^p$ and various other approximation
properties. We prove that the AP of Haagerup and Kraus (see definition
\ref{def=AP_groupe}) implies $\SCHURCBAP{p}$ for any
$1<p<\infty$. We also prove that for such $p$, if the non-commutative
$L^p$-space associated to a discrete group has the OAP (or the
stronger property CBAP), then this group has the property
$\SCHURCBAP{p}$. When $G$ is hyperlinear, these results are
consequences of \cite{MR1971296}. Here we prove these results without
the hypothesis of hyperlinearity. Since we are working in $S^p$
instead of general non-commutative $L^p$-spaces, we are able to adapt
the argument of \cite{MR1971296} and give elementary proofs that avoid
some technicalities (in particular we avoid the use of the results
from the unpublished work \cite{junge}). The results in this section
are however certainly well-known to experts, and the proofs standard.
We also prove that, for discrete groups and $1<p<\infty$,
$\CASch{p}{G}$ can only take the two values $1$ or $\infty$. All the
aforementioned results are corollaries of a same result (Theorem
\ref{thm=approxiimplCB}) on the approximation, in the stable
point-norm topology (see below for definitions), of the identity on a
Schatten class.

For a discrete group $G$, we denote by $\tau_G$ the usual tracial state on
the von Neumann algebra of $G$, and by $L^p(\tau_G)$ the corresponding
non-commutative $L^p$ space (for $1 \leq p \leq \infty$). 

Before we give precise statements and proofs we have to recall some
basic facts on the stable point-norm topology.

\subsection{The stable point-norm topology}

For an operator space $V$, we recall the definition of the \emph{stable
  point-norm} topology $\mathcal T_{\mathrm n}$ on $CB(V,V)$~: $\mathcal
T_{\mathrm n}$ is the weakest topology making the seminorms $T\mapsto \|id
\otimes T(x)\|$ for $x \in \mathcal K(\ell^2) \otimes_{\mathrm{min}} V =
S^\infty[V]$ continuous. In this section we use the notation $S^\infty[V]$
for $\mathcal K(\ell^2) \otimes_{\mathrm{min}} V$.
 
We recall the definition of OAP, which was given in the introduction~:
\begin{dfn}
  An operator space $V$ has the operator space approximation property (OAP)
  if the identity on $V$ belongs to the $\mathcal T_{\mathrm n}$-closure of
  the space $F(V,V)$ of finite rank operators on $V$.
\end{dfn}

We wish to study this notion when $V$ is a non-commutative $L^p$-space
$L^p(\mathcal M,\tau)$. Non-commutative $L^p$ spaces indeed have a natural
operator space structure but, as explained in subsection
\ref{sec=cb_sur_LpNC}, this structure is more simply described in terms of
$L^p(B(\ell^2) \overline \otimes \mathcal M,Tr \otimes \tau)$ ($Tr$ denotes
the usual semi-finite trace on $B(\ell^2)$). Lemma
\ref{lem=equivalent_topology} below will allow us to give a simpler
equivalent definition of the topology $\mathcal T_{\mathrm{n}}$ in
Definition \ref{dfn=def_equival_T_n}.

Lemma \ref{lem=equivalent_topology} is a characterization of the topology
$\mathcal T_{\mathrm n}$, in terms of vector-valued Schatten classes
$S^p[V]$ defined in \cite{MR1648908}. Except in the following two lemmas,
in the remaining of the paper the notation $S^p[V]$ will only be used when
$V=S^p(H)$ or $V=L^p(\tau_G)$ for a discrete group $G$. In this case the
space $S^p[V]$ coincides with $S^p(\ell^2 \otimes H)$ or (if $p<\infty$)
$L^p(Tr \otimes \tau)$.

\begin{lemma}\label{lem=equivalent_topology}
Let $1 \leq p \leq \infty$. The topology $\mathcal T_{\mathrm n}$ on $CB(V,V)$
coincides with the topology defined by the family of seminorms $T
\mapsto \sup_i \|id_{S^p} \otimes T (x_i)\|_{S^p[V]}$, for all
$(x_i)_{i \geq 0} \in c_0(S^p[V])$.
\end{lemma}
\begin{rem}
  We view $S^p$ as the increasing union of $S^p_n$, $n\geq 1$.

  Let us denote by $\mathcal T_{\mathrm n}^p$ the topology described in this
  lemma.  Since $\cup_n S^p_n[V]$ is dense in $S^p[V]$, this topology
  $\mathcal T_{\mathrm n}^p$ coincides with the topology defined by the
  seminorms $T \mapsto \sup_i \|id_{S^p} \otimes T (x_i)\|_{S^p[V]}$
  for $(x_i)_{i \geq 0} \in c_0(\cup_n S^p_n[V])$.  We will use this
  elementary fact in the proof below.
\end{rem}
\begin{proof}
  We first consider the case $p=\infty$ (note that by definition,
  $S^\infty[V] = \mathcal K(\ell^2) \otimes_{\mathrm{min}}V$). The inclusion
  $\mathcal T_{\mathrm n} \subset \mathcal T_{\mathrm n}^\infty$ is obvious. The other direction
  is classical and follows very easily from the fact that $\mathcal
  K(\ell^2) \otimes_{\mathrm{min}} \mathcal K(\ell^2) \otimes_{\mathrm{min}} V =
  \mathcal K(\ell^2 \otimes_2 \ell^2) \otimes_{\mathrm{min}} V$. Indeed if $x_i
  \in \mathcal K(\ell^2) \otimes V$ converges to $0$, then $x =
  \oplus_{i} x_i$ belongs to $\mathcal K(\ell^2) \otimes_{\mathrm{min}}
  \mathcal K(\ell^2) \otimes_{\mathrm{min}} V$, and for any $T \in CB(V,V)$,
  $\|id \otimes T (x)\| = \sup_i \|id \otimes T (x_i)\|$.

  Assume now $p<\infty$. We prove first that $\mathcal T_{\mathrm n}^p \subset
  \mathcal T_{\mathrm n}^\infty$.  Take $(x_i)_{i \geq 0} \in c_0(S^p[V])$. By the
  properties of $S^p[V]$ (Theorem 1.5 in \cite{MR1648908}), $x_i$ can
  be written as $x_i=a_i \cdot v_i \cdot b_i$ with $a_i, b_i$ in the
  unit ball of $S^{2p}$ and $\|v_i\|_{S^\infty[V]} \leq 2
  \|x_i\|_{S^p[V]}$ ($\|x_i\|_{S^p[V]}$ is in fact equal to the
  infimum of $\|v_i\|_{S^\infty[V]}$ over all such decompositions). In
  particular, $\lim_i \|v_i\|_{S^\infty[V]}=0$, and moreover for any
  $T \in CB(V,V)$ and $n \in \N$, $\|id \otimes T(x_i)\|_{S^p[V]} \leq
  \|id \otimes T(v_i)\|_{S^\infty[V]}$.

  For the reverse inclusion $\mathcal T_{\mathrm n}^\infty \subset \mathcal T_{\mathrm n}^p$, we
  use the above remark for $p = \infty$. Let us consider $x_i \in
  M_{n_i}(V)$ such that $\|x_i\|_{M_{n_i}(V)} \to 0$. By Lemma 1.7
  in \cite{MR1648908}, we have that 
\[ \|x_i\|_{M_{n_i}(V)} = \sup \left\{ \|a x_i b\|_{S^p_{n_i}[V]}, a,b
  \textrm{ in the unit ball of }S^{2p}_{n_i}\right\}.\]
Consider a sequence $(y_{i,j})_{j \geq 0}$ in the ball of radius
$\|x_i\|_{M_{n_i}(V)}$ in $S^p_{n_i}[V]$ converging to $0$ and such
that for any $a,b$ in the unit ball of $S^{2p}_{n_i}$, $\|a x_i
b\|_{S^p_{n_i}[V]}$ belongs to the closed convex hull of $\{y_{i,j},j
\geq 0\}$. Then $\lim_{|i|+|j| \to \infty} \|y_{i,j}\|_{S^p[V]} = 0$
and for any $T \in CB(V,V)$, and $a,b$ in the unit ball of
$S^{2p}_{n_i}$, we have that
\[ \|(id \otimes T)(a x_i b)\|_{S^p_{n_i}[V]} \leq \sup_j \|(id
\otimes T)(y_{i,j})\|_{S^p_{n_i}[V]}.\]
Hence, 
\begin{eqnarray*}\sup_i \|(id \otimes T) (x_i)\|_{M_{n_i}(V)} 
& \leq &   \sup_i \sup_j \|(id \otimes T)(y_{i,j})\|_{S^p_{n_i}[V]}.
\end{eqnarray*}
This concludes the proof of $\mathcal T_{\mathrm n}^\infty \subset \mathcal T_{\mathrm n}^p$ and of
the Lemma.
\end{proof}

When $V=S^p$ or $V=L^p(\tau_G)$ (or more generally $V=L^p(\mathcal M,\tau)$
for a semi-finite normal faithful trace $\tau$ on $\mathcal M$), Lemma
\ref{lem=equivalent_topology} shows that the definition of the topology
$\mathcal T_{\mathrm{n}}$ and of the property $OAP$ is equivalent to the
following definition, which has the advantage not to rely on the precise
definition of the operator space structure on $V$. In this definition $G$
is a discrete group, and $H$ a Hilbert space.
\begin{dfn}\label{dfn=def_equival_T_n} Let $1 \leq p \leq \infty$.

  The topology $\mathcal T_{\mathrm{n}}$ on $CB(S^p(H),S^p(H))$ is the
  weakest topology making the seminorms $T\mapsto \sup_i\|id \otimes
  T(x_i)\|_{S^p(\ell^2 \otimes H)}$ for $(x_i)_{i\geq 0}\in
  c_0(S^p(\ell^2\otimes H))$ continuous.

  If $p<\infty$ the topology $\mathcal T_{\mathrm{n}}$ on
  $CB(L^p(\tau_G),L^p(\tau_G))$ is the weakest topology making the
  seminorms $T\mapsto \sup_i\|id \otimes T(x_i)\|_{L^p(Tr \otimes \tau_G)}$
  for $(x_i)_{i\geq 0}\in c_0(L^p(Tr \otimes \tau_G))$ continuous.

  $L^p(\tau_G)$ has OAP if the identity on $L^p(\tau_G)$ is in the
  $\mathcal T_{\mathrm{n}}$-closure of the space of finite rank operators
  on $L^p(\tau_G)$. 
\end{dfn}
The reader unfamiliar with the notions of vector-valued $S^p$ can start
with this definition, forget Lemma \ref{lem=equivalent_topology} which will
not be used later, and take in Lemma
\ref{lem=equivalent_topology_convexcase}, $V=S^p(H)$ of $L^p(\tau_G)$, so
that $S^p[V]$ is elementary.

Since the weak closure and the norm closure of a convex set coincide,
we even get~:
\begin{lemma}\label{lem=equivalent_topology_convexcase} Let $C$ be a convex
  subset of $CB(V,V)$, $u \in CB(V,V)$ and $1\leq p \leq \infty$. Then $u$
  belongs to the $\mathcal T_{\mathrm n}$-closure of $C$ if and only if for any $a
  \in S^p[V]$ and $b \in S^p[V]^*$, $\langle b, id \otimes u(a)\rangle$
  belongs to the closure of
\[ \left\{\langle b, id \otimes T(a)\rangle, T \in C \right\}.\]
\end{lemma}
\begin{proof} By Lemma \ref{lem=equivalent_topology}, $u$ belongs to the
  $\mathcal T_{\mathrm n}$-closure of $C$ if and only if for any $(x_i)_{i \geq 0}
    \in c_0(S^p[V])$, $(id \otimes u(x_i))_{i \geq 0}$ belongs to the norm
    closure in $c_0(S^p[V])$ of
    \[ \left\{ \left(id \otimes T)(x_i) \right)_{i \geq 0}, T \in
      C\right\}.\] Since this latter set is convex, this is equivalent to
    saying that $((id \otimes u)(x_i))_{i \geq 0}$ belongs to its weak
    closure, \emph{i.e.}  that $\sum_i \langle b_i,(id\otimes
    u)(x_i)\rangle$ belongs to the closure of
\[    \left\{ \sum_i \langle b_i , (id \otimes T)(x_i)\rangle, T \in
      C\right\}\]
 for every $b_i \in (S^{p}[V])^*$ such that $\sum_i
  \|b_i\|_{(S^{p}[V])^*} < \infty$. Fix such $(x_i)_{i} \in c_0(S^p[V])$
  and $(b_i)_i \in \ell^1(S^p[V]^*)$. We now construct $\widetilde b \in
  S^p[V]^*$ and $\widetilde x \in S^p[V]$ such that for any $T \in
  CB(V,V)$, 
  \begin{equation}\label{eq=convexite}
\sum_i \langle b_i , (id \otimes T)(x_i)\rangle = \langle \widetilde b,(id
\otimes T)(\widetilde x)\rangle
\end{equation} 
This will conclude the proof. Let $\lambda_i = \|b_i\|_{(S^{p'}[V])^*}$,
and $\widetilde b_i = \lambda_i^{-1/p} b_i$ (with $0^{-1/p}0=0$) and
$\widetilde x_i = \lambda_i^{1/p} x_i$. If $X$ denotes the space $\ell^{p}
(S^{p}[V])$ (if $p<\infty$) or $ c_0 (S^{\infty}[V])$ (if $p=\infty$), we
therefore have that $\widetilde b = (\widetilde b_i)_{i \geq 0} \in
\ell^{p'} (S^{p}[V]^*) \simeq X^*$ and $\widetilde x= (\widetilde x_i)_{i
  \geq 0} \in X$. Note that the space $X$ is naturally contained in
$S^p(\ell^2 \otimes \ell^2)[V]$ as a complemented subspace. Indeed, if
$p<\infty$, $\ell^p(S^p)$ is naturally embedded in $S^p(\ell^2 \otimes
\ell^2)$, and there is a completely positive projection $P: S^p(\ell^2
\otimes \ell^2) \to \ell^p(S^p)$ (the conditional expectation). By
\cite{MR1324838}, Theorem 0.1, $P \otimes id_V$ extends to a bounded map on
the vector-valued spaces. The same proof holds for $p=\infty$. The element
$\widetilde a \in X^*$ therefore defines an element in the dual of
$S^{p}(\ell^{2} \otimes \ell^2)[V]$ (by $x \mapsto \widetilde a(P x)$), and
with these identifications, \eqref{eq=convexite} is easy to check.
\end{proof}

\subsection{AP for groups and approximation on $S^p$}
For facts on AP (Haagerup's and Kraus' approximation property) for discrete
groups, see \cite{MR2391387}, Appendix D. For a discrete group $G$ and a
function $\varphi:G \to \C$ we denote by $m_\varphi$ the corresponding
Fourier multiplier on $C^*_{\mathrm{red}}(G)$ defined by $m_\varphi
\lambda(s) =\varphi(s) \lambda(s)$. Recall that we denote also by $M_{\chap
  \varphi}$ the corresponding Schur multiplier.

\begin{dfn}\label{def=AP_groupe} A discrete group $G$ is said to have the
  approximation property (AP) if there is a net $\varphi_\alpha$ of
  functions from $G$ to $\C$ with finite support and such that for any $a
  \in \mathcal K(\ell^2) \otimes_{\mathrm{min}} C^*_{\mathrm{red}}(G)$ and $f \in
  L^1(Tr \otimes \tau_G)$,
\[\lim_\alpha \langle f, id \otimes m_{\varphi_\alpha} (a) \rangle =
\langle a,f\rangle.\]
\end{dfn}
\begin{rem}\label{rk=AP_implique_approxcompact}
  The AP for a discrete group $G$ implies that $id_{\mathcal K(\ell^2G)}$
  belongs to the $\mathcal T_{\mathrm n}$-closure in $CB(\mathcal K(\ell^2 G),
  \mathcal K(\ell^2 G))$ of
\[\{M_{\chap \varphi}, \varphi:G \to \C \textrm{ of finite support.} \}.\]
\end{rem}
\begin{proof}
  By Lemma \ref{lem=equivalent_topology_convexcase}, we have to prove that
  for any $a \in \mathcal K(\ell^2) \otimes_{\mathrm{min}} \mathcal K(\ell^2 G)] = \mathcal K(\ell^2 \otimes
  \ell^2 G)$, and $b \in \mathcal K(\ell^2 \otimes \ell^2 G)^*= S^1(\ell^2
  \otimes \ell^2 G)$, $\langle a,b \rangle$ belongs to the closure of 
  \[\{\langle b, (id \otimes M_{\chap \varphi}) (a) \rangle, \varphi
  \textrm{ of finite support}\}.\] (we choose to denote by $\langle
  a,b\rangle$ the duality bracket $Tr(a b)$).

  To do this consider the trace-preserving embedding $i:\mathcal K(\ell^2
  G) \to \mathcal K(\ell^2 G) \otimes_{\mathrm{min}} C^*_{\mathrm{red}}(G)$ defined
  on the dense subspace spanned by the elementary matrices $e_{s,t}$ for
  $s,t \in G$ by $i(e_{s,t}) = e_{s,t} \otimes \lambda(s^{-1}t)$. Let $E$
  be the conditional expectation. Then $E \circ id \otimes m_\varphi \circ
  i$ corresponds to the Schur multiplier with symbol $\chap \varphi$. Hence
  for $a \in \mathcal K(\ell^2 \otimes \ell^2 G)$ and $b \in S^1(\ell^2
  \otimes \ell^2 G)$,
  \[ \langle b, id \otimes M_{\chap \varphi} (a) \rangle = \langle (id
  \otimes i)(b), (id \otimes m_\varphi)\circ (id \otimes i)(a)\rangle.\]
  Since $id \otimes i (a)$ (resp. $id \otimes i(b)$) belongs to $\mathcal
  K(\ell^2 \otimes \ell^2(G)) \otimes_{\mathrm{min}} C^*_{red}(G)$ (resp. $L^1(Tr
  \otimes \tau_G)$, where $Tr$ denotes the usual trace on $B(\ell^2 \otimes
  \ell^2 G)$), this proves the claim.
\end{proof}

Combining the above proof and the proof in \cite{MR1220905} that the OAP
for $C^*_{\mathrm{red}}(G)$ implies the AP for $G$ (the same idea was
already used in \cite{haagerup2}, Theorem 2.6, to prove that the CBAP for
$C^*_{\mathrm{red}}(G)$ implies the weak amenability for $G$), we get the
following Proposition~:
\begin{prop}\label{thm=OAP_implique_MultSchur}
  Let $G$ be a discrete group and $1\leq p <\infty$. If $L^p(\tau_G)$ has
  the OAP, then the identity on $S^p(\ell^2 G)$ belongs to the $\mathcal
  T_{\mathrm n}$-closure of the space
  \[\left \{ M_{\chap \varphi}, \varphi:G \to \C \textrm{ of finite
      support}\right\}.\]
\end{prop}
\begin{proof}
  We use again Lemma \ref{lem=equivalent_topology_convexcase}. Since (see
  the proof of Lemma \ref{lemma=def_equivalente_par_AG}) the space $\left
    \{ M_{\chap \varphi}, \varphi:G \to \C \textrm{ of finite
      support}\right\}$ is norm-dense (for the cb-norm of linear maps on
  $S^p(\ell^2G)$) in $\left \{ M_{\chap \varphi}, \varphi \in
    A(G)\right\}$, it is in fact enough to prove that for any $a \in
  S^p[S^p(\ell^2 G)] = S^p(\ell^2 \otimes \ell^2 G)$, and $b \in S^p(\ell^2
  \otimes \ell^2 G)^*$, $\langle a,b \rangle$ belongs to the closure of
  \[\{\langle b, (id \otimes M_{\chap \varphi}) (a) \rangle, \varphi \in
  A(G)\}.\] 

  For any finite rank map $T:L^p(\tau_G) \to L^p(\tau_G)$ define
  $\varphi_T:G \to \C$ by
\begin{equation}\label{eq=def_de_phi_apartirT}\varphi_T(g) = \tau\left( T (\lambda(g))
    \lambda(g)^*\right).\end{equation} 
We claim that $\varphi_T \in A(G)$. We even prove that $\varphi_T \in
\ell^2(G)$. To prove this we can assume that $T$ has rank one, \emph{i.e.}
is of the form $x \mapsto \xi(x) a$ for some $\xi \in L^p(\tau_G)^*$ and $a
\in L^p(\tau_G)$. Then $\varphi_T(g) = \xi(\lambda(g)) \tau_G(a
\lambda(g)^*)$. If $p \geq 2$, then $|\xi(\lambda(g))| \leq \|\xi\|$ and
  \[\|\left(\tau_G(a \lambda(g)^*)\right)_g\|_{\ell^2(G)} =
  \|a\|_{L^2(\tau_G)} \leq \|a\|_{L^p(\tau_G)}.\] 
  If $p<2$ then since $L^p(\tau_G)^* \sim L^{p'}(\tau_G)$ with $1/p'+1/p=1$
  (note $p'>2$), the previous computation implies that
  $(\xi(\lambda(g)))_g$ belongs to $\ell^2(G)$ and $\tau_G(a \lambda(g)^*)$
  is bounded.

We now prove that for any $a \in S^p[S^p(\ell^2 G)] = S^p(\ell^2 \otimes \ell^2 G)$, and $b \in S^p(\ell^2
  \otimes \ell^2 G)^*$, $\langle a,b \rangle$ belongs to the closure of
  \[\{\langle b, (id \otimes M_{\chap \varphi_T}) (a) \rangle, T \in
  F(L^p(\tau_G), L^p(\tau_G))\}.\] For simplicity of notation we prove the
  case $p>1$. Then $S^p(\ell^2 \otimes \ell^2 G)^* = S^{p'}(\ell^2 \otimes
  \ell^2 G)$. The proof for $p=1$ is the same, except that $S^{p'}(\ell^2
  \otimes \ell^2 G)$ has to be replaced by $B(\ell^2 \otimes \ell^2
  G)$. The inclusion $i$ in the proof of Remark
  \ref{rk=AP_implique_approxcompact} induces a completely contractive map
  (that we still denote by the same letter) $i:S^p(\ell^2 G) \to L^p(Tr
  \otimes \tau_G)$. Here $Tr$ denotes the usual semi-finite trace on
  $S^p(\ell^2 G)$. The same holds for $p'$. Moreover we have, for $a \in
  S^p(\ell^2 \otimes \ell^2 G)$ and $b \in S^{p'}(\ell^2 \otimes \ell^2
  G)$,
  \[ \langle b, id \otimes M_{\chap \varphi_T} (a) \rangle = \langle (id
  \otimes i)(b), (id \otimes T)\circ (id \otimes i)(a)\rangle.\] But $(id
  \otimes i)(a)$ belongs to $S^p(\ell^2 \otimes \ell^2G)[L^p(\tau_G)]$ and
  $(id \otimes i)(b)$ belongs to its dual space $S^{p'}(\ell^2 \otimes
  \ell^2G)[L^{p'}(\tau_G)]$. Therefore, by the assumption that $L^p(\tau_G)$
  has the OAP and by Lemma \ref{lem=equivalent_topology_convexcase},
  $\langle b,a\rangle = \langle (id \otimes i)(b), (id \otimes T)\circ (id
  \otimes i)(a)\rangle$ belongs to the closure of
\[\left\{ \langle (id \otimes i)(b), (id \otimes T)\circ (id \otimes
  i)(a)\rangle, T \in F(L^p(\tau_G), L^p(\tau_G))\right\}.\]
This proves the Proposition.
\end{proof}

The proof of the following Proposition is very close to the proof of
Theorem 1.1 in \cite{MR1971296}. In fact this Proposition also follows from
Theorem 1.1 in \cite{MR1971296} and from Proposition
\ref{thm=OAP_implique_MultSchur}.
\begin{prop}\label{prop=APimpliqueAPp}
  Let $G$ be a discrete group with AP and $1<p<\infty$. Then the identity
  on $S^p(\ell^2 G)$ belongs to the $\mathcal T_{\mathrm n}$-closure of the
  space
  \[\left \{ M_{\chap \varphi}, \varphi:G \to \C \textrm{ of finite
      support}\right\}.\]
\end{prop}
\begin{proof}
  Denote $H = \ell^2 \otimes \ell^2 G$. By Lemma
  \ref{lem=equivalent_topology_convexcase}, it is enough to prove that for
  any $a \in S^{p}(H)$ and $b \in S^{p'}(H)$, $Tr(ab) = \langle a,b\rangle$
  belongs to the closure of
\[\left\{ \langle a,(id \otimes M_{\chap \varphi})(b)\rangle, \varphi:G \to \C \textrm{
    with finite support}\right\}.\]

We prove this using the complex variable. We use the notation $S^\infty(H)
= \mathcal K(H)$. Let $S$ be the strip $\{z \in \C, 0<Re(z)<1\}$ and
consider maps $f$, $g$ in $C_0(\overline S;S^\infty(H))$ that are
holomorphic on $S$, such that $f(1/p)=a$, $g(1/p)=b$ and such that $t
\mapsto f(1+it)$ belongs to $C_0(\R; S^1(H))$ and $t \mapsto g(it)$ belongs
to $C_0(\R; S^1(H))$. Such maps exist because $S^p(H)$ coincides with the
complex interpolation space $[S^\infty(H),S^1(H)]_{1/p}$, but they can be
constructed explicitely. To construct $f$, write $a= a_0 a_1$ with $a_0 \in
S^\infty(H)$ and $a_1$ a positive element in $S^p(H)$, and take
$f(z)=e^{(z-1/p)^2} a_0 a_1^{pz}$. In the same way, write $b= b_0 b_1$ with
$b_0 \in S^\infty(H)$ and $b_1$ a positive element in $S^{p'}(H)$, and take
$g(z)=e^{(z-1/p)^2} b_0 b_1^{p'(1-z)}$.

Then the set $K=0 \cup \{ g(1+it)^T,t \in \R \} \cup \{ f(it),t \in \R \}$
is a compact subset of $S^\infty(H)$ ($\cdot ^T$ denotes the transpose
map). It is classical that any compact subset containing $0$ in a Banach
space is contained in the closed convex hull of a sequence converging to
$0$. By the assumption that $G$ has AP and by Remark
\ref{rk=AP_implique_approxcompact}, for any $\varepsilon>0$, there is a
$\varphi:G \to \C$ of finite support such that for any $x \in K$
\[ \|(id \otimes M_{\chap \varphi}) x - x \|_{S^\infty(H)}<\varepsilon.\] In
particular, if $Re(z)=0$
\[ \left|\langle g(z),(id \otimes M_{\chap \varphi}) f(z) -
  f(z)\rangle\right| < \varepsilon \|g(z)\|_{S^1(H)}.\] In the same way,
since $Tr( x (id \otimes M_{\varphi})(y)) = Tr( y^T (id \otimes
M_{\varphi})(x^T))$, we get that for $Re(z)=1$ 
\[ \left|\langle g(z),(id \otimes M_{\chap \varphi}) f(z) -
  f(z)\rangle\right| < \varepsilon \|f(z)\|_{S^1(H)}.\] By the maximum
principle, if $C=\max(\sup_{t \in R} \|g(it)\|_{S^1(H)},\sup_{t \in R}
\|f(1+it)\|_{S^1(H)} )$, we get that for $z=1/p$,
\[ \left|\langle b,(id \otimes M_{\chap \varphi}) a\rangle - \langle
  b,a\rangle\right| < \varepsilon C.\]
Since $\varepsilon$ is arbitrary, this concludes the proof. 
\end{proof}

\subsection{Different approximation properties on $S^p$}

The main result of this section is the following Theorem (and its
corollaries). This is in the same spirit as the theorem of
Grothendieck which states that for a separable dual Banach space, the
approximation property implies the metric approximation property. Its
proof is an adaptation of Grothendieck's argument to the stable
topology.

\begin{thm}\label{thm=approxiimplCB}
  Let $H$ be a Hilbert space and let $F_0$ be a subspace of the space
  $F(S^p(H),S^p(H))$ of bounded finite rank operators on $S^p(H)$, such that
  $id_{S^p}$ belongs to the $\mathcal T_{\mathrm n}$-closure of $F_0$. Then
  $id_{S^p}$ belongs to the $\mathcal T_{\mathrm n}$-closure of $\{T \in
  F_0, \|T\|_{cb} \leq 1\}$.
\end{thm}
Before we give the proof of this Theorem, let us state three corollaries.
\begin{corollary}\label{coro=CASch1ouRien}
  If $G$ is a discrete group and $1<p<\infty$, then $\CASch{p}{G}=1$ or
  $\CASch{p}{G}=\infty$.
\end{corollary}
\begin{proof}
  Note that $\CASch{p}{G}\leq c$ if and only if $id_{S^p(\ell^2G)}$ belongs
  to the $\mathcal T_{\mathrm n}$-closure in $CB(S^p(\ell^2G),S^p(\ell^2G))$ of $\{
  M_{\chap \varphi},\varphi:G \to \C \textrm{ with finite support}\} \cap
  \{T, \|T\|_{cb} \leq c\}$. This Corollary therefore follows from Theorem
  \ref{thm=approxiimplCB} applied for the space $F_0$ consisting of the
  $M_{\chap \varphi}$ for all $\varphi:G \to \C$ with finite support.
\end{proof}

\begin{corollary}\label{coro=AP_implique_notrePte}
  If $G$ is a discrete group with AP and $1<p<\infty$, then
  $\CASch{p}{G}=1$.
\end{corollary}
\begin{proof} This follows from Proposition \ref{prop=APimpliqueAPp} and
  from Theorem \ref{thm=approxiimplCB} applied for the space $F_0$
  consisting of the $M_{\chap \varphi}$ for all $\varphi:G \to \C$ with
  finite support.
\end{proof}
\begin{corollary}\label{coro=OAP_Lp_implieque_notrePte}
  If $1<p<\infty$ and $G$ is a discrete group such that $L^p(\tau_G)$ has
  the OAP (or the CBAP), then $\CASch{p}{G}=1$.
\end{corollary} 
\begin{proof}
  The CBAP is stronger than OAP. If $L^p(\tau_G)$ has the OAP, then by
  Proposition \ref{thm=OAP_implique_MultSchur}, the hypothesis in Theorem
  \ref{thm=approxiimplCB} holds with the space $F_0$ consisting of the
  $M_{\chap \varphi}$ for all $\varphi:G \to \C$ with finite support. This
  implies that $\CASch{p}{G}=1$. 
\end{proof}

The main tool in the proof of Theorem \ref{thm=approxiimplCB} will be the
following Lemma, which expresses (in the vocabulary of \cite{MR1793753},
chapter 12) that a completely integral map on $S^p$ is completely
nuclear. Junge proved in the unpublished paper \cite{junge} that this holds
for any non-commutative $L^p$-space (with $1<p<\infty$) of a QWEP von
Neumann algebra. We give an elementary statement and an elementary proof,
due to Gilles Pisier~:
\begin{lemma}\label{lem=integralImpliquenucleaire}
  Let $1<p<\infty$ and let $H_1$, $H_2$ be Hilbert spaces and $\Psi$ a
  linear map of norm less than $1$ on $F(S^p(H_1),S^p(H_2))$ equipped with
  the completely bounded norm. Then there exist $x \in S^p(\ell^2\otimes
  H_1)$ and $y \in S^{p'}(\ell^2\otimes H_2)$ satisfying $\|x\|_p
  \|y\|_{p'} < 1$ and such that
\[\Psi(T) = \langle y, (id \otimes T)(x)\rangle \textrm{ for any }T \in F(S^p(H_1),S^p(H_2)).\]
In particular, $\Psi$ extends to a $\mathcal T_{\mathrm n}$-continuous linear map on
$CB(S^p(H_1),S^p(H_2))$ of norm less than $1$.
\end{lemma}
\begin{proof}
For any linear map $\Psi: F(S^p(H_1),S^p(H_2)) \to \C$, denote
\begin{gather*}
N_1(\Psi) = \sup_{T \in F(S^p(H_1),S^p(H_2)), \|T\|_{cb} \leq 1}
|\Psi(T)|\\ N_2(\Psi)=\inf_{x\in S^p(\ell^2\otimes H_1),y \in
  S^{p'}(\ell^2\otimes H_2)}\|x\|_p \|y\|_{p'},
\end{gather*}
where the infimum is taken over all $x,y$ satisfying $\Psi(T)=\langle y, id
\otimes T(x)\rangle$ for all $T \in F(S^p(H_1),S^p(H_2))$.  For $i=1,2$,
$N_i$ is a norm which makes $\{\Psi, N_i(\Psi)<\infty\}$ a Banach space,
and obviously $N_1 \leq N_2$. We prove that $N_1=N_2$. When $H_1$ or $H_2$
is finite dimensional, this is classical and very easy~: namely, for $i=1$
or $2$, the space $\{\Psi: F(S^p(H_1),S^p(H_2)) = CB(S^p(H_2),S^p(H_2)) \to
\C\textrm{ linear bounded}\}$ coincides (as a vector space) with $S^p(H_1)
\otimes S^p(H_2)^*$, and when equipped with the norm $N_i$, its dual space
is naturally $CB(S^p(H_1),S^p(H_2))$ with the norm $\|\cdot\|_{cb}$.

If $K$ is a closed subspace of $H_1$, denote by $e_K \in B(H_1)$ the
orthogonal projection on $K$ and $P_K:A \in S^p(H_1) \mapsto e_K A e_K \in
S^p(H_1)$. Denote also by $\Psi_K$ the map $T \in F(S^p(H_1),S^p(H_2)) \to
\Psi(T P_K)$. By the case $\dim(H_1)<\infty$, we have that $N_1(\Psi_K) =
N_2(\Psi_K)$ for any finite dimensional subspace $K$ of $H_1$. If
$\{0\}=K_0 \subset K_1\subset K_2 \subset \dots K_N$ is an increasing
family of orthogonal finite dimensional subspaces of $H_1$ and if
$q=\max(p,2)$, we claim that
\begin{equation}\label{eq=restriction_de_Cauchy} \frac 1 2\left(\sum_{n=1}^N N_2(\Psi_{K_n} - \Psi_{K_{n-1}})^q\right)^{1/q} \leq
N_2(\Psi_{K_N}) = N_1(\Psi_{K_N}) \leq N_1(\Psi).\end{equation} The middle
equality has already been proved, and the second inequality is obvious. The
first inequality follows from the following inequality valid for any $x \in
S^p(\ell^2 \otimes H_1)$:
\[ \left(\sum_{n=1}^N \|(id \otimes P_n - id \otimes P_{n-1})(x)\|_p^q
\right)^{1/q} \leq 2 \|x\|_p,\] which follows from the inequalities,
valid for any family $(q_n)_{n \geq 1}$ of orthogonal projections on
$\ell^2 \otimes H_1$
\begin{gather*}
 \left(\sum_{n=1}^N \| q_n x\|_p^q \right)^{1/q} \leq \|x\|_p\\
\left(\sum_{n=1}^N \| x q_n\|_p^q \right)^{1/q} \leq
\|x\|_p.\end{gather*} When $p \geq 2$ this can be proved using the triangle
inequality in $S^{p/2}$. When $p=1$, this can be proved using the fact that
the unit ball in $S^1$ is the closed convex hull of rank one operators, and
for $p<2$, this follows by interpolation between $p=1$ and $p=2$.

\eqref{eq=restriction_de_Cauchy} then implies that the net $(\Psi_K)$ (for
$K$ a finite dimensional subspace of $H_1$) is Cauchy for $N_2$,
\emph{i.e.} for any $\varepsilon$ there exists a finite dimensional
subspace $K_\varepsilon$ such that for any finite dimensional $K$
containing $K_\varepsilon$, $N_2(\Psi_{K} - \Psi_{K_\varepsilon}) <
\varepsilon$. This implies that it converges for the norm $N_2$ to an
element of $N_2$-norm not greater than $N_1(\Psi)$. This limit is $\Psi$,
which shows that $N_2(\Psi) \leq N_1(\Psi)$ and which concludes the proof
of $N_2 = N_1$.

The second statement of the Lemma is then immediate, because for $x \in
S^p(\ell_2 \otimes H_1)$ and $y \in S^{p'}(\ell^2\otimes H_2)$, the
formula $T \mapsto \langle y, (id \otimes T)(x) \rangle$ defines a
$\mathcal T_{\mathrm n}$-continuous map on $CB(S^p(H_1), S^p(H_2))$.
\end{proof}

\begin{proof}[Proof of Theorem \ref{thm=approxiimplCB}]
  This proof relies on the Hahn-Banach Theorem. For convenience we denote
  $S^p(H)$ simply by $S^p$. Let $\Phi:CB(S^p,S^p) \to \C$ be a $\mathcal
  T_{\mathrm n}$-continuous linear form such that $|\Phi(T)| \leq 1$ for all $T \in
  F_0$ with $\|T\|_{cb}\leq 1$ (equivalently $|\Phi(T)| \leq \|T\|_{cb}$
  for all $T \in F_0$). The aim is to prove that $|\Phi(id_V)| \leq 1$. For
  this we show that for any $\varepsilon>0$, $\Phi$ coincides on the space
  $F_0$ with a linear map $\Psi$ on $CB(S^p,S^p)$, which is also $\mathcal
  T_{\mathrm n}$-continuous and for which $\|\Psi\| \leq 1+\varepsilon$. This would
  conclude the proof because then $\Psi=\Phi$ on the $\mathcal T_{\mathrm n}$-closure
  of $F_0$, and in particular $\Phi(id_{S^p}) = \Psi(id_{S^p})$ is less
  than $1+\varepsilon$.

  The restriction of $\Phi$ to $F_0$ is of norm $1$. By Hahn-Banach it
  extends to a norm $1$ functional $\Phi_1$ on $F(S^p,S^p)$. By Lemma
  \ref{lem=integralImpliquenucleaire}, for any $\varepsilon>0$, $\Phi_1$
  extends to a $\mathcal T_{\mathrm n}$-continuous map $\Psi$ on $CB(S^p,S^p)$ of
  norm less than $1+\varepsilon$.
\end{proof}

\section{Case of $SL_{r+1}(F)$}
\label{sect=Contre-exemple}
The aim of this section is to prove Theorem
\ref{thm=main_thm_continu}. This is done at the end of this section, as a
consequence of Proposition \ref{prop-unites-approchees}. 

Let $p>2$.  Let $n\in \N^{*}$ such that $p>2+\frac{2}{n}$.  Set
\[\eps=n(\frac{1}{2}-\frac{1}{p})-\frac{1}{p}=\frac{n}{2p}\big( p-(2+\frac{2}{n})\big)\in \R_{+}^{*}.\]
Let $r\in \N^{*}$ such that $r\geq 2n$, $F$ be a non-archimedian local
field and $\mathcal O$ its ring of integers. Let $G=SL_{r+1}(F)$ and
$K=SL_{r+1}(\mathcal O)$ which is a maximal compact subgroup of $G$.

\begin{prop}\label{prop-unites-approchees}
  The constant function $1$ on $G$ cannot be approximated (for the topology
  of uniform convergence on compact subsets) by functions $f$ in $C_0(G)$
  such that $\|\chap f\|_{MS^p(L^2(G))}$ is bounded uniformly. In particular,
\[\CASch{p}{SL_{r+1}(F)} = \infty.\]
\end{prop}

This proposition follows from
\begin{prop}\label{prop-unites-approchees-Kinv}
  The constant function $1$ on $G$ cannot be approximated (for the topology
  of uniform convergence on compact subsets) by $K$-biinvariant functions
  $f$ in $C_0(G)$ such that $\|\chap f\|_{MS^p(L^2(G))}$ is bounded
  uniformly.
\end{prop}
\begin{proof}[Proof of Proposition \ref{prop-unites-approchees-Kinv} using
 Proposition \ref{prop-unites-approchees}]
Averaging on the left and on the right by $K$ one sees that it is enough to
show that one cannot approximate $1$ by $K$-biinvariant functions in
$C_0(G)$ uniformly bounded for $\|\chap f\|_{MS^p(L^2(G))}$. 
\end{proof}

Let $\pi$ be a uniformizer of $\OO$, and let $\OO^\times$ denote the units
(or invertibles) of $\OO$. Denote by $\F = \OO/\pi\OO$ the residue field of
$F$.  To define an absolute value $|\cdot|$ on $F$ we have to choose
$|\pi|\in (0,1)$. Then $|\cdot|$ is defined in the following way~:
$|x|=|\pi|^{\lam}$ if $x\in \pi^{\lam}\OO^{\times}$ for $\lam \in \Z$ and
$|x|=0$ if $x=0$. The standard choice is to take $|\pi|=q^{-1}$, because
with this choice $d(xa)=|x| da$ for any $x \in F$, where $da$ denotes a
Haar measure on $F$. Since we do not use this property, we prefer to keep
the choice of $|\pi|\in (0,1)$ arbitrary.  The coefficients of the matrices
below are easier to understand if they are written as powers of $\pi^{-1}$
instead of powers of $\pi$.  To keep the size of matrices reasonnable we
introdude the notation $e=\pi^{-1}$, so that $|e|=|\pi|^{-1}$ is an
arbitrary number in $ (1,\infty)$.  The important property of $|\cdot|$ is
that it is non-archimedian, \emph{i.e.} the triangle inequality has the
stronger form $|x+y|\leq \max(|x|,|y|)$ for any $x,y \in F$.

\begin{rem}
  The reader unfamiliar with these notions can consider the special case where $q$ is a prime 
  number and  $F = \Q_q$ (we avoid the usual notation $\Q_p$ because
  the letter $p$ is already used). Note that $\Q_q$ is the field obtained
  by completion of $\Q$ for the distance given by the absolute value on
  $\Q$, $|a/b|=|q|^{v_q(a)-v_q(b)}$, where $|q|\in (0,1)$ is arbitrary and
  $v_q(a)$ is the greatest $k$ such that $q^k$ divides $a$ (the resulting
  field does not depend on the choice of $|q|\in (0,1)$). In the special case where  $F = \Q_q$, $\mathcal O$
  is $\Z_q$, the unit ball in $\Q_q$ (or equivalently the closure of $\Z$),
  a convenient choice for $\pi$ is to simply take $\pi=q$ and the residue
  field is $\Z/q\Z$.
\end{rem}
 
Let \begin{gather*}\Lambda=\{(\lam_{1},...,\lam_{r})\in \N^{r},
  \lam_{1}\geq \lam_{2}-\lam_{1}\geq \lam_{3}-\lam_{2} \geq ... \geq \lam_{r}-\lam_{r-1} \geq -\lam_{r}
  \}.\end{gather*} 
  
For $(\lam_{1},...,\lam_{r})\in \N^{r}$ denote by
$P(\lam_{1},...,\lam_{r})$ the polygon whose vertices are the points
$(i,\lam_{i})$ for $i\in \{0,...,r+1\}$, setting $\lam_{0}=0$ and
$\lam_{r+1}=0$. Then $\Lambda$ is the set of $(\lam_{1},...,\lam_{r})\in
\N^{r}$ such that $P(\lam_{1},...,\lam_{r})$ is convex (or equivalently
such that the piecewise affine map on $[0,r+1]$ taking values $\lambda_i$
on $i$ is concave). The $\lam_{i+1}-\lam_{i}$ for $i\in \{0,...,r\}$ are
the slopes of the polygon and $2\lam_{i}-\lam_{i-1}-\lam_{i+1}$ is called
\brisure{} at vertex $i$, for $i\in \{1,...,r\}$.  A polygon is convex if
all its \brisures{} are nonnegative. The picture below gives an example for
$r=4$.

  \ifx\JPicScale\undefined\def\JPicScale{1}\fi
\unitlength \JPicScale mm
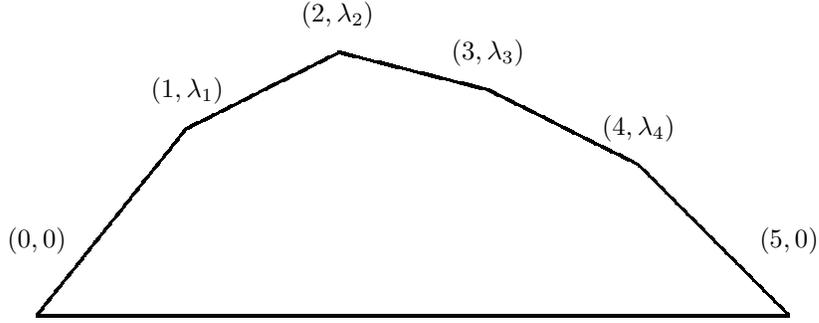
\begin{picture}(110,50)(0,5)
\linethickness{0.3mm}
\put(10,10){\line(1,0){100}}
\linethickness{0.3mm}
\multiput(10,10)(0.12,0.15){167}{\line(0,1){0.15}}
\linethickness{0.3mm}
\multiput(30,35)(0.24,0.12){83}{\line(1,0){0.24}}
\linethickness{0.3mm}
\multiput(50,45)(0.48,-0.12){42}{\line(1,0){0.48}}
\linethickness{0.3mm}
\multiput(70,40)(0.24,-0.12){83}{\line(1,0){0.24}}
\linethickness{0.3mm}
\multiput(90,30)(0.12,-0.12){167}{\line(1,0){0.12}}
\put(10,20){\makebox(0,0)[cc]{$(0,0)$}}

\put(110,20){\makebox(0,0)[cc]{$(5,0)$}}

\put(30,40){\makebox(0,0)[cc]{$(1,\lam_1)$}}

\put(50,50){\makebox(0,0)[cc]{$(2,\lam_2)$}}

\put(70,45){\makebox(0,0)[cc]{$(3,\lam_3)$}}

\put(90,35){\makebox(0,0)[cc]{$(4,\lam_4)$}}

\end{picture}

For $(\lam_{1},...,\lam_{r})\in  \Lambda$ denote 
\[D(\lam_{1},...,\lam_{r})=\begin{pmatrix} e^{\lam_{1}} & 0 & 0& \dots & 0 \\
  0 & e^{\lam_{2}-\lam_{1}} & \ddots& \ddots & \vdots \\
  0 & \ddots & \ddots & \ddots & 0 \\
  \vdots & \ddots & \ddots& e^{\lam_{r}-\lam_{r-1}}& 0 \\
  0 & \dots & 0& 0 & e^{-\lam_{r}} \\
\end{pmatrix}\in G,\]
where the exponants of $e$ are the slopes of the polygon
$P(\lam_{1},...,\lam_{r})$.
                               
The map associating $KD(\lam_{1},...,\lam_{r})K$ to
$(\lam_{1},...,\lam_{r})\in \Lambda$ induces a bijection between $\Lambda$
et $K\backslash G/K$.
 
For a matrix $A=(a_{kl})$ denote $\|A\|=\max(|a_{kl}|)$. Then for $A\in G$,
 \begin{gather}\nonumber A
 \in KD(\lam_{1},...,\lam_{r})K
 \\
 \label{double-norme}
 \textrm{ if and only if }\|\Lambda^{i}A\|=|e|^{\lam_{i}}\textrm{ for all } i\in
 \{1,...,r\} .\end{gather}
  More concretely $\|\Lambda^{i}A\|$ is the maximum of the norms  of all $i \times i$-minors of $A$. 
When $A\in KD(\lam_{1},...,\lam_{r})K$ one says that
$P(\lam_{1},...,\lam_{r})$ is the  polygon of $A$. The reason why we introduce these polygons is that the $\lambda_{i}$ are more convenient parameters than the slopes $\lambda_{i+1}-\lambda_{i}$ (see \eqref{double-norme} above and lemma~\ref{lemTrenf} below) and that the convexity condition satisfied by the $\lambda_{i}$ is best seen by drawing the polygon. 

Denote by $B$ the Borel subgroup of $G$ (formed of upper-triangular matrices). 

\begin{prop}
  For any function $f\in C_{c}(G)$, let $g=f\vert_B \in C_c(B)$ be the
  restriction of $f$ to $B$. Then \[\|\chap
  g\|_{MS^p(L^{2}(B))}\leq \|\chap f\|_{MS^{p}(L^{2}(G))}.\]

  If $f$ is $K$-biinvariant it is an equality: $\|\chap
  g\|_{MS^p(L^{2}(B))} = \|\chap f\|_{MS^{p}(L^{2}(G))}$.
\end{prop}
\begin{rem} The notation $\chap f$ was introduced at the beginning of
  section \ref{sect=approximation}. Note that by Theorem
  \ref{thm=mult_de_symb_continu}, the norms of all the multipliers
  appearing in this proposition are equal to their cb-norms.
\end{rem}
\begin{proof}
  For $p=\infty$ this is proved in proposition 1.6 of~\cite{haagerup3}. 

  For general $p$ it is a consequence of the results in section
  \ref{sec=Schur_multi}.  The first point follows from Theorem
  \ref{thm=mult_de_symb_continu}. Moreover since $B$ and $G$ are both
  without isolated points (and the Haar measure has full support) Theorem
  \ref{thm=mult_de_symb_continu} implies that $\|\chap g\|_{MS^p(L^{2}(B))}
  = \|\chap g\|_{cbMS^p(L^{2}(B))}$ and $\|\chap f\|_{MS^{p}(L^{2}(G))} =
  \|\chap f\|_{cbMS^{p}(L^{2}(G))}$. But by Lemma
  \ref{lem=cnnorm_inv_par_tribu}, since $G/K=B/(B \cap K)$, both terms
  $\|\chap g\|_{cbMS^p(L^{2}(B))}$ and $\|\chap f\|_{cbMS^{p}(L^{2}(G))}$
  are equal to $\|\chap f\|_{cbMS^{p}(L^{2}(G/K))}$.
\end{proof}

\begin{lemma}\label{lemTrenf}
  There is a constant $C$ such that for all $K$-biinvariant $f\in
  C_{c}(G)$, for $(\lam_{1},...,\lam_{r})\in \Lambda$ and $i\in
  \{1,...,r\}$ such that
\[(\lam_{1},...,\lam_{i-1},\lam_{i}+1,\lam_{i+1},...,\lam_{r})\in \Lambda
,\] one has
  \begin{gather}\nonumber \big|  f(D(\lam_{1},...,\lam_{r}))-f(D(\lam_{1},...,\lam_{i-1},\lam_{i}+1,\lam_{i+1},...,\lam_{r}))
                               \big|   
                                     \\
                                     \label{est-Trenf1}
\leq C q^{-\eps (2\lam_{i+1}-\lam_{i}-\lam_{i+2})}
                               \|\chap f\|_{MS_{p}(L^{2}(G))}
                                                          \text{\  if \ \ }
                               r-i\geq n 
        \\ \nonumber 
        \text{and \ \ } \big|  f(D(\lam_{1},...,\lam_{r}))-f(D(\lam_{1},...,\lam_{i-1},\lam_{i}+1,\lam_{i+1},...,\lam_{r}))
                               \big|   
                                     \\
                                     \label{est-Trenf2}
\leq C q^{-\eps (2\lam_{i-1}-\lam_{i-2}-\lam_{i})}
                               \|\chap f\|_{MS_{p}(L^{2}(G))}
                                                          \text{\ \ for \ }
                               i-1\geq n. \end{gather}                  
\end{lemma}

The following lemma is very close to Lemma 5.5 
 in~\cite{anniversaire} (and of the estimates following). 
 
 Let $m\in \N^{*}$.  For $k\in \{0,...,m\}$ let us denote by \[T_{k}=
 ((T_{k})_{(a_{1},...,a_{n},b),(x_{1},...,x_{n},y)})_{(a_{1},...,a_{n},b)\in
   (\OO/\p^{m}\OO)^{n+1}, (x_{1},...,x_{n},y)\in (\OO/\p^{m}\OO)^{n+1}}\] the
 matrix defined by
\begin{eqnarray*}
(T_{k})_{(a_{1},...,a_{n},b),(x_{1},...,x_{n},y)}&=&q^{-mn} \textrm{ if } 
y=\sum_{i=1}^{n}a_{i}x_{i}+b+\p^k
\textrm{\ \ in\ \ } \OO/\p^{m}\OO, \\
&=&0\textrm{ otherwise.}
\end{eqnarray*}

\begin{lemma}\label{lem-Tkk-1}
One has
\begin{gather}\label{ineg1-lem-Tkk-1}
\|T_{m}-T_{m-1}\|_{S_{p}}\leq 2
q^{-\eps m}\end{gather} and for $u,v\in \C$ one has \begin{gather}\label{ineg2-lem-Tkk-1}
\|uT_{m}-vT_{m-1}\|_{S_{p}}\geq |u-v|.\end{gather}
\end{lemma}
\begin{proof}
  Since  \[(T_{m})_{(a_{1},...,a_{n},b),(x_{1},...,x_{n},y)}\text{ and
  }(T_{m-1})_{(a_{1},...,a_{n},b),(x_{1},...,x_{n},y)}\] only depend on
  $y-b$ one has
\begin{gather*}\|T_{m}-T_{m-1}\|_{S_{p}}^{p}
\\
=\sum_{\eta\in \widehat{\OO/\p^{m}\OO}}
\big|1-\eta(\p^{m-1}) \big|^{p} 
\Big\|q^{-mn} (\eta(\sum_{i=1}^{n}a_{i}x_{i}))_{(a_{1},...,a_{n}),(x_{1},...,x_{n})\in ( \OO/\p^{m}\OO)^{n}}  \Big\| _{S_{p}}^{p}
\\
=\sum_{\eta\in \widehat{\OO/\p^{m}\OO}}
\big|1-\eta(\p^{m-1}) \big|^{p}
\Big\|q^{-m} (\eta(ax))_{a,x\in  \OO/\p^{m}\OO}  \Big\| _{S_{p}}^{pn}.
\end{gather*}
If $1-\eta(\p^{m-1})\neq 0$ one has $|1-\eta(\p^{m-1})|\leq 2 $ and $\eta$
is a nondegenerated character of $\OO/\p^{m}\OO$. But for such a character 
\[
\big\| q^{-m}(\eta(ax))_{a,x\in \OO/\p^{m}\OO} \big\| _{S_{p}}=
q^{-\frac{m}{2}+\frac{m}{p}}\] because the
matrix \[q^{-\frac{m}{2}}(\eta(ax))_{a,x\in \OO/\p^{m}\OO} \] is unitary (as
a matrix of a Fourier transform).  But there are exactly
$(1-\frac{1}{q})q^{m}$ non degenerated characters of $\OO/\p^{m}\OO$. One
thus has
\begin{gather*}\|T_{m}-T_{m-1}\|_{S_{p}}^{p}
\leq  (1-\frac{1}{q})q^{m} 2^{p} q^{-\frac{npm}{2}+nm}
\leq  2^{p}q^{(1-\frac{np}{2}+n)m}=2^{p}q^{-p\eps m},
\end{gather*}
which proves \eqref{ineg1-lem-Tkk-1}.  

The inequality \eqref{ineg2-lem-Tkk-1} holds because the vector in
$\ell^{2}((\OO/\pi^m \OO)^{n+1})$ with coordinates all equal to $1$ is an
eigenvector for $T_m$ and $T_{m-1}$ with eigenvalue $1$. Hence it is an
eigenvector of $uT_m - vT_{m-1}$ with eigenvalue $u-v$.
\end{proof}
 
The following Lemma is a rephrasing of Theorem \ref{thm=mult_de_symb_continu}.
\begin{lemma}\label{M0AG-Schur}
Let $k\in \N$, $A\in M_{k}(\C)$, $H$ a locally compact group,
$f\in C_{c}(H)$ and  $\alpha, \beta:\{1,...,k\}\to H$ two injective maps. Then 
\[\big\|\big(f(\alpha(i)\beta(j))A_{ij}\big)_{i,j\in
  \{1,...,k\}}\big\|_{S^p}\leq \|\chap{f}\|_{MS^p(L^{2}(H))}\|A\|_{S^p}.\]
\end{lemma}
\begin{proof} Theorem \ref{thm=mult_de_symb_continu} implies this
  with $f(\alpha(i)^{-1}\beta(j))$ instead of $f(\alpha(i)\beta(j))$, but
  the two versions are equivalent.
\end{proof}

We use a combination of the two preceding lemmas.
\begin{lemma}\label{lem-combine}
  Let $m\in \N^{*}$. Let $H$ be a locally compact group and $f\in
  C_{c}(H)$.  Let $\alpha,\beta :(\OO/\p^{m}\OO)^{n+1}\to H$ be two injective
  applications and $u,v\in \C$ such that
\begin{gather}\label{cond1-lem-combine}
f(\alpha(a_{1},...,a_{n},b)\beta(x_{1},...,x_{n},y))=u \text{ if } 
y=\sum_{i=1}^{n}a_{i}x_{i}+b
\text{ in } \OO/\p^{m}\OO
\\\label{cond2-lem-combine}
 f(\alpha(a_{1},...,a_{n},b)\beta(x_{1},...,x_{n},y))=v \text{  if  } 
y=\sum_{i=1}^{n}a_{i}x_{i}+b+\p^{m-1}
\text{  in  } \OO/\p^{m}\OO.
\end{gather}
Then $|u-v|\leq 2 q^{-\eps m} \|\chap f\|_{MS_{p}(L^{2}(H))}$.
\end{lemma}
\begin{proof}
  By Lemma~\ref{M0AG-Schur} applied to $A=T_{m}-T_{m-1}$, one has $\|u
  T_{m}-v T_{m-1}\|_{S_{p}}\leq \|\chap f\|_{MS_{p}(L^{2}(H))}\|T_{m}-
  T_{m-1}\|_{S_{p}}$. One then applies the inequalities
  \eqref{ineg1-lem-Tkk-1} and \eqref{ineg2-lem-Tkk-1} of
  Lemma~\ref{lem-Tkk-1}. 
\end{proof}

\begin{proof}[Proof of Lemma ~\ref{lemTrenf}.] 
The estimate  (\ref{est-Trenf2}) can be deduced from the estimate
(\ref{est-Trenf1}) by the automorphism 
\[\theta :  A \mapsto \begin{pmatrix}    
  0   & \dots & 0 & 1  \\
  \vdots  & \diagup & \diagup  &  0\\
 0  &\diagup &\diagup &\vdots  \\
 1  & 0& \dots & 0
   \end{pmatrix} {}^{t}A^{-1}\begin{pmatrix}    
  0   & \dots & 0 & 1  \\
  \vdots  & \diagup & \diagup  &  0\\
 0  &\diagup &\diagup &\vdots  \\
 1  & 0& \dots & 0
\end{pmatrix} \] of $G$, which preserves $K$ and $B$. Indeed
$\theta(D(\lam_{1},...,\lam_{r}))=D(\lam_{r},...,\lam_{1})$. It is thus
enough to prove (\ref{est-Trenf1}).

Let
$(\lam_{1},...,\lam_{r})\in \Lambda$ and $i\in \{1,...,r-n\}$ such that
\begin{gather}\label{concavite+1}(\lam_{1},...,\lam_{i-1},\lam_{i}+1,\lam_{i+1},...,\lam_{r})\in \Lambda .\end{gather}

Set $\lam_{0}=0$ and $\lam_{r+1}=0$. Denote by $\mu_{1},...,\mu_{r+1}$ the
slopes of the polygon $P(\lam_{1},...,\lam_{r})$, \emph{i.e.}
$\mu_{i}=\lam_{i}-\lam_{i-1}$. Since  $(\lam_{1},...,\lam_{r})\in \Lambda$
one has $\mu_{1}\geq \mu_{2}\geq ... \geq \mu_{r+1}$ and moreover
$\sum_{i=1}^{r+1}\mu_{i}=0$. The condition \eqref{concavite+1} is
equivalent to 
\begin{gather}\label{concavite+1-bis} \mu_{i-1}> \mu_{i} \text{ \ \ and \ \ }
\mu_{i+1}>  \mu_{i+2}\end{gather}
because the slopes of the polygon \[P(\lam_{1},...,\lam_{i-1},\lam_{i}+1,\lam_{i+1},...,\lam_{r})\]
are \[(\mu_{1},..., \mu_{i-1}, \mu_{i}+1, \mu_{i+1}-1, \mu_{i+2} ... ,  \mu_{r+1}). \]
 
We are going to apply Lemma~\ref{lem-combine}
with \begin{gather}\label{def-m} H=B \text{ \ and \ }
  m=\mu_{i+1}-\mu_{i+2}=2\lam_{i+1}-\lam_{i}-\lam_{i+2}\in
  \N^{*}. \end{gather} In other words, $m$ is the \brisure{} of
$P(\lam_{1},...,\lam_{r})$ at vertex $i+1$.

Let us fix a section $\sigma:\OO/\p^{m}\OO\to \OO$ of the projection $\OO\to
\OO/\p^{m}\OO$. The choice of this section has no importance. 

Let us define two maps
$\alpha,\beta :(\OO/\p^{m}\OO)^{n+1}\to B$ (where $B$ is the subgroup of
upper-triangular matrices in $SL_{r+1}$) 
in the following way~:  
\begin{gather*}\alpha(a_{1},...,a_{n},b)= \\
\begin{pmatrix} e^{\mu_{1}} & 0 & \dots& \dots & \dots&\dots&0 \\ 
  0 & \ddots& \ddots& \ddots & \ddots &\ddots&\vdots\\ 
  \vdots & \ddots & e^{\mu_{i-1}}  & \ddots & \ddots &\ddots&\vdots\\ 
        \vdots & \ddots & \ddots& \alpha'(a_{1},...,a_{n},b)& \ddots &\ddots&\vdots\\ 
          \vdots & \ddots & \ddots& \ddots & e^{\mu_{i+n+2}} &\ddots&\vdots\\         
            \vdots & \ddots & \ddots& \ddots & \ddots & \ddots&0\\ 
              0 & \dots & \dots& \dots & \dots &0&e^{\mu_{r+1}}         
                       \end{pmatrix}
                       \\ \text{ and \ \ \ }
                       \beta(x_{1},...,x_{n},y)= \\
\begin{pmatrix} 1 & 0 & \dots& \dots & \dots&\dots&0 \\ 
  0 & \ddots& \ddots& \ddots & \ddots &\ddots&\vdots\\ 
  \vdots & \ddots & 1  & \ddots & \ddots &\ddots&\vdots\\ 
        \vdots & \ddots & \ddots& \beta'(x_{1},...,x_{n},y)& \ddots &\ddots&\vdots\\ 
          \vdots & \ddots & \ddots& \ddots & 1 &\ddots&\vdots\\         
            \vdots & \ddots & \ddots& \ddots & \ddots & \ddots&0\\ 
              0 & \dots & \dots& \dots & \dots &0&1         
                       \end{pmatrix}
\end{gather*}
where the matrices are block-diagonal with all blocks of size $1$ except
the blocks $\alpha'(a_{1},...,a_{n},b)$ and $\beta'(x_{1},...,x_{n},y)$
which are square matrices of size $n+2$. The position of the block
$\beta'(x_{1},...,x_{n},y)$ is the same as the position of the block
$\alpha'(a_{1},...,a_{n},b)$, so that
\begin{gather*}\alpha(a_{1},...,a_{n},b)\beta(x_{1},...,x_{n},y) = \\
\begin{pmatrix} e^{\mu_{1}} & 0 & \dots& \dots & \dots&\dots&0 \\ 
  0 & \ddots& \ddots& \ddots & \ddots &\ddots&\vdots\\ 
  \vdots & \ddots & e^{\mu_{i-1}}  & \ddots & \ddots &\ddots&\vdots\\ 
        \vdots & \ddots & \ddots& \alpha'(a_{1},...,a_{n},b)\beta'(x_{1},...,x_{n},y)& \ddots &\ddots&\vdots\\ 
          \vdots & \ddots & \ddots& \ddots & e^{\mu_{i+n+2}} &\ddots&\vdots\\         
            \vdots & \ddots & \ddots& \ddots & \ddots & \ddots&0\\ 
              0 & \dots & \dots& \dots & \dots &0&e^{\mu_{r+1}}         
                       \end{pmatrix}.
\end{gather*}
The matrices $\alpha'(a_{1},...,a_{n},b)$ and $\beta'(x_{1},...,x_{n},y)$
are defined by
\begin{gather*}\alpha'(a_{1},...,a_{n},b)= \\
\begin{pmatrix} e^{\mu_{i}} & -e^{\mu_{i}}\sigma(a_{1}) & -e^{\mu_{i}}\sigma(a_{2})& \dots &-e^{\mu_{i}}\sigma(a_{n}) &-e^{\mu_{i}+\mu_{i+1}-\mu_{i+2}}\sigma(b) \\ 
  0 & e^{\mu_{i+2}}& 0& \dots & \dots &0\\ 
  \vdots & \ddots & e^{\mu_{i+3}}  & \ddots &\ddots&\vdots\\ 
        \vdots & \ddots & \ddots&  \ddots&\ddots &\vdots\\ 
          \vdots & \ddots & \ddots& \ddots & e^{\mu_{i+n+1}} &0\\         
                    0 & \dots & \dots& \dots & 0 &e^{\mu_{i+1}}         
                       \end{pmatrix}
                       \\ \text{ and \ \ \ }
                     \beta'(x_{1},...,x_{n},y)= \\
\begin{pmatrix} 1 & 0& \dots & \dots &0&e^{\mu_{i+1}-\mu_{i+2}}\sigma(y) \\ 
  0 & 1& \ddots& \ddots & \vdots &e^{\mu_{i+1}-\mu_{i+2}}\sigma(x_{1})\\ 
  \vdots & \ddots & 1 & \ddots &\vdots&e^{\mu_{i+1}-\mu_{i+2}}\sigma(x_{2})\\ 
        \vdots & \ddots & \ddots&  \ddots&0 &\vdots\\ 
          \vdots & \ddots & \ddots& \ddots & 1 &e^{\mu_{i+1}-\mu_{i+2}}\sigma(x_{n})\\         
                    0 & \dots & \dots& \dots & 0 &1         
                       \end{pmatrix}
 \end{gather*}
Let us compute
\begin{gather*}\alpha'(a_{1},...,a_{n},b)\beta'(x_{1},...,x_{n},y)=\\
\begin{pmatrix} e^{\mu_{i}} & -e^{\mu_{i}}\sigma(a_{1}) & -e^{\mu_{i}}\sigma(a_{2})& \dots &-e^{\mu_{i}}\sigma(a_{n}) &e^{\mu_{i}+\mu_{i+1}-\mu_{i+2}}w \\ 
  0 & e^{\mu_{i+2}}& 0& \dots & \dots &e^{\mu_{i+1}}\sigma(x_{1})\\ 
  \vdots & \ddots & e^{\mu_{i+3}}  & \ddots &\ddots&e^{\mu_{i+1}-\mu_{i+2}+\mu_{i+3}}\sigma(x_{2})\\ 
        \vdots & \ddots & \ddots&  \ddots&\ddots &\vdots\\ 
          \vdots & \ddots & \ddots& \ddots & e^{\mu_{i+n+1}} &e^{\mu_{i+1}-\mu_{i+2}+\mu_{i+n+1}}\sigma(x_{n})\\         
                    0 & \dots & \dots& \dots & 0 &e^{\mu_{i+1}}         
                       \end{pmatrix}
 \end{gather*}
 with
 $w=\sigma(y)-\big(\sum_{i=1}^{n}\sigma(a_{i})\sigma(x_{i})+\sigma(b)\big)\in
 \OO$.

We are going to study the following cases~: 
\begin{itemize}
\item i) 
  $w=0 \mod \p^{m}\OO$ 
  \item ii) 
   $w=\p^{m-1} \mod  \p^{m}\OO$.
   \end{itemize} 

   Since $m=\mu_{i+1}-\mu_{i+2}$ by \eqref{def-m} and $e=\pi^{-1}$, one has
  \begin{gather*}e^{\mu_{i}+\mu_{i+1}-\mu_{i+2}}w\in
    \p^{-\mu_{i}}\OO\text{ \  in case i)} \\
    \text{ and \ } e^{\mu_{i}+\mu_{i+1}-\mu_{i+2}}w\in
    \p^{-\mu_{i}-1}+\p^{-\mu_{i}}\OO \text{ \ in case ii)}.\end{gather*}
  Since $\mu_{i}\geq \mu_{i+1}\geq ... \geq \mu_{i+n+1}$, it follows that
  in case i)
  \[\|\alpha'(a_{1},...,a_{n},b)\beta'(x_{1},...,x_{n},y)\|=|e|^{\mu_{i}}\]
  whereas in case ii) 
  \[\|\alpha'(a_{1},...,a_{n},b)\beta'(x_{1},...,x_{n},y)\|=|e|^{\mu_{i}+1}.\]
  Thanks to the second inequality in \eqref{concavite+1-bis} one checks
  that in both cases, for all $j\in \{2,...,n+2\}$,
 \[\big\|\Lambda^{j} \big(\alpha'(a_{1},...,a_{n},b)\beta'(x_{1},...,x_{n},y)\big)\big\|=|e|^{\mu_{i}+\mu_{i+1}+...+\mu_{i+j-1}}. \]
As a consequence, 
$\alpha'(a_{1},...,a_{n},b)\beta'(x_{1},...,x_{n},y)$ belongs to  
\[GL_{n+2}(\OO)\begin{pmatrix} e^{\mu_{i}} & 0 & \dots& \dots&0 \\ 
  0 & e^{\mu_{i+1}}& \ddots& \ddots & \vdots\\ 
  \vdots & \ddots & e^{\mu_{i+2}}  & \ddots & \vdots\\ 
            \vdots & \ddots & \ddots& \ddots &0\\ 
              0 & \dots & \dots &0&e^{\mu_{i+n+1}}         
                       \end{pmatrix}
GL_{n+2}(\OO)\]
in case i) and to  
\[GL_{n+2}(\OO)\begin{pmatrix} e^{\mu_{i}+1} & 0 & \dots& \dots&0 \\ 
  0 & e^{\mu_{i+1}-1}& \ddots& \ddots & \vdots\\ 
  \vdots & \ddots & e^{\mu_{i+2}}  & \ddots & \vdots\\ 
            \vdots & \ddots & \ddots& \ddots &0\\ 
              0 & \dots & \dots &0&e^{\mu_{i+n+1}}         
                       \end{pmatrix}
GL_{n+2}(\OO)\]
in case ii).

Thanks to condition
\eqref{concavite+1-bis}, it follows that  
$\alpha(a_{1},...,a_{n},b)\beta(x_{1},...,x_{n},y)$ belongs to  
\begin{gather*}KD(\lam_{1},...,\lam_{r})K\text{ \ 
in case i) }
\\
\text{  and to \ }KD(\lam_{1},...,\lam_{i-1},\lam_{i}+1,\lam_{i+1},...,\lam_{r})K
\text{ \  
 in case ii).}\end{gather*} The hypotheses of Lemma~\ref{cond1-lem-combine}
are therefore satisfied with
 \begin{gather*}H=B, 
 u=f(D(\lam_{1},...,\lam_{r}))\text{ and }v=f(D(\lam_{1},...,\lam_{i-1},\lam_{i}+1,\lam_{i+1},...,\lam_{r})).\end{gather*} 
This concludes the proof of Lemma~\ref{lemTrenf}.
\end{proof}
 
For all $m\in \N^{*}$ denote $\lam^{m}=(\lam^{m}_{1}, ..., \lam^{m}_{r})\in
\Lambda$ the element defined by $\lam^{m}_{i}=mi(r+1-i)$. Note that all the
\brisures{} of the associated polygon are equal to $2m$. One has
 \[D(\lam^{m})=\begin{pmatrix} e^{mr} & 0 & \dots& \dots&0 \\ 
  0 & e^{m(r-2)}& \ddots& \ddots & \vdots\\ 
  \vdots & \ddots & \ddots & \ddots & \vdots\\ 
            \vdots & \ddots &  \ddots &e^{m(2-r)}& 0\\ 
              0 & \dots & \dots &0&e^{m(-r)}        
                       \end{pmatrix}\]
 
 \begin{lemma}\label{lemTrenf-sauts}
   There is a constant $C$ such that for all $K$-biinvariant function $f\in
   C_{c}(G)$, for all $m\in \N^{*}$ one has
   \begin{gather} \label{est-Trenf-sauts}\big| f(D(\lam^{m})) \big| \leq C
     q^{-2\eps m } \|\chap f\|_{MS_{p}(L^{2}(G))} . \end{gather}
\end{lemma} 
\begin{proof}[Proof of Lemma ~\ref{lemTrenf-sauts}.] It is enough to prove
  that there exists $C$ such that for all $K$-biinvariant function $f \in
  C_c(G)$, for all $m\in \N^*$ one has
\[\big| f(D(\lam^{m})) - f(D(\lam^{m+1}))\big| \leq C
q^{-2\eps m } \|\chap f\|_{MS_{p}(L^{2}(G))} .\] 
This inequality follows from Lemma~\ref{lemTrenf}. One can indeed pass from
$\lam^{m}$ to $\lam^{m+1}$ by $\sum_{i=1}^{r}i(r+1-i)$ successive
transformations consisting in increasing by $1$ the $i^{\text{th}}$
coefficient and letting the others fixed. One applies
 
  \noindent \eqref{est-Trenf1} if $i\leq \frac{r+1}{2}$ (which implies that
  $i\leq r-n$ thanks to the hypothesis $r\geq 2n+1$)
 
  \noindent and \eqref{est-Trenf2} if $i\geq \frac{r+1}{2}$ (which implies
  that $i-1\geq n$ thanks to the hypothesis $r\geq 2n+1$).
  
  Moreover one can manage the keep all the \brisures{} $\geq 2m-2$. If $C$
  is the constant in Lemma~\ref{lemTrenf} one thus gets 
\[\big|  f(D(\lam^{m}))-f(D(\lam^{m+1}))
                               \big|   
                      \leq C\big(\sum_{i=1}^{r}i(r+1-i)\big) q^{-\eps (2m-2) }
                               \|\chap f\|_{MS_{p}(L^{2}(G))}\]
and this concludes the proof of Lemma~\ref{lemTrenf-sauts}.
\end{proof}
 
\begin{proof}[Proof of Proposition ~\ref{prop-unites-approchees-Kinv}.] If is
  an immediate consequence of Lemma ~\ref{lemTrenf-sauts}.
\end{proof}
\begin{rem}
In \eqref{est-Trenf-sauts}, the function $m \mapsto f(D(\lam^m))$ is
exponentially small when $m \to \infty$ whereas the proof of Haagerup in
\cite{haagerup2} (in the case $G=SL_2(\R) \ltimes \R^2$ and $p=\infty$)
does not imply such a result. For more on this, see \cite{vincentremarque}.
\end{rem}

We are now able to prove the main results of the introduction in the
non-archimedian case.

\begin{proof}[Proof of Theorem \ref{thm=main_thm_continu}.]
The statement for $p>2+2/n$ is an immediate consequence of Proposition
\ref{prop-unites-approchees}. If $p<2-2/(n+2)$, notice that $p'>2+2/n$
if $p'$ is the conjugate exponent of $p$~: $1/p+1/p'=1$. Proposition
\ref{prop-unites-approchees} and hence Theorem
\ref{thm=main_thm_continu} therefore also hold, by Remark
\ref{rem=duality}.
\end{proof}

\begin{proof}[Proof of Theorem \ref{thm=main_theorem}.]
By Theorem \ref{thm=main_thm_continu} and Theorem
\ref{thm=characterzation}, $\Gamma$ does not have
$\SCHURCBAP{p}$. The statement follows from Corollary
\ref{coro=OAP_Lp_implieque_notrePte}.
\end{proof}

\begin{proof}[Proof of Theorem \ref{thm=main_theorem_AP} (non-archimedian case).]
If $4<p<\infty$, as in the proof above, $\Gamma$ does not have
$\SCHURCBAP{p}$. The theorem thus from Corollary
\ref{coro=AP_implique_notrePte}.
\end{proof}

\section{Case of $SL_{r}(\R)$}
\label{sect=Contre-exemple_R}

This section is devoted to the proof of Theorem
\ref{thm=main_thm_continu_R} and its consequences. This will be deduced at
the end of this section from the following Proposition.

\begin{prop}\label{prop=SLrR} Let $r \geq 3$ and $G=SL_r(\R)$. Let $1 \leq
  p \leq\infty$ such that $p>4$ or $p<4/3$.  The constant function $1$ on
  $G$ cannot be approximated (for the topology of uniform convergence on
  compact subsets) by functions $f$ in $C_0(G)$ such that $\|\chap
  f\|_{MS^p(L^2(G))}$ is bounded uniformly~:
\[\CASch{p}{SL_{r}(\R)} = \infty.\]
\end{prop}

This main tool to prove the Proposition is 
\begin{lemma}\label{lemma=SL3R_main_estimate}
Let $G= SL_3(\R)$, $K=SO_3(\R)$, and $4<p\leq\infty$. Let $0<\varepsilon <
1/2-2/p$. There is a constant $C>0$ such that for any $K$-biinvariant
function $\varphi \in C_0(G)$, and any $t>0$
\[ \left|\varphi\begin{pmatrix} e^t & 0 & 0 \\ 0 &1&0\\0&0&e^{-t}
\end{pmatrix}\right| \leq C e^{-\varepsilon t} \|\chap \varphi\|_{MS^p(L^2(G))}.\]
\end{lemma}
We first deduce Proposition \ref{prop=SLrR} from this Lemma.
\begin{proof}[Proof of Proposition \ref{prop=SLrR}.]
  Lemma \ref{lemma=SL3R_main_estimate} implies that, if $4<p\leq\infty$,
  the function $1$ on $SL_3(\R)$ cannot be approximated (pointwise) by
  $SO_3(\R)$-biinvariant functions such $f$ in $C_0(G)$ such that $\|\chap
  f\|_{MS^p(L^2(G))}$ is bounded uniformly. By the same averaging argument
  as in the proof of Proposition \ref{prop-unites-approchees}, we deduce
  Proposition \ref{prop=SLrR} in the case $r=3$ and $p>4$. The case $r=3$
  and $p<4/3$ follows from Remark \ref{rem=duality}.

For $r>3$, the map
\[ A \mapsto \begin{pmatrix} A & 0\\0& 1_{r-3}\end{pmatrix}\]
realizes $SL_3(\R)$ as a closed subgroup of $SL_r(\R)$. Theorem
\ref{thm=mult_de_symb_continu} implies that  Proposition
\ref{prop=SLrR} holds also for $r>3$.
\end{proof}

Lemma \ref{lemma=SL3R_main_estimate} is proved as in section
\ref{sect=Contre-exemple}, using the same techniques as in the proof of
strong property $(T)$ for $SL_3(\R)$ in \cite{duke}. From now on we fix
$G$, $K$, $p>4$ and $\varepsilon>0$ as in Lemma
\ref{lemma=SL3R_main_estimate}. We use some notation and facts from
\cite{duke}, section $2$. We denote by $\sphere$ the unit sphere in $\R^3$,
equipped with its usual probability measure denoted by $dx$. For any
$\delta \in [-1,1]$, we denote by $T_\delta$ the operator on $L^2(\sphere)$
defined, for a continuous function $f:\sphere \to \C$ in the following way
(and extended by continuity to a norm $1$ operator). If $x \in \sphere$,
$T_\delta f(x)$ is the average of $f$ on the circle $\{y \in \sphere,
\langle x,y\rangle=\delta\}$. We first state the analogue of Lemma
\ref{lem-Tkk-1} of this paper.
\begin{lemma}\label{lem-Tkk-1_R}
There is a constant $C_1$ such that for $\delta \in [-1/2,1/2]$
\[ \left\|T_0 - T_\delta\right\|_{S^p(L^2(\sphere))} \leq C_1
|\delta|^{1/2-2/p}.\]
Moreover for any $a,b \in \C$,
\[ \left\|a T_0 - b T_\delta\right\|_{S^p(L^2(\sphere))} \geq |a-b|.\]
\end{lemma}
\begin{proof}[Sketch of proof]
Let $P_n$ be the $n$-th Legendre polynomial normalized by $P_n(1)=1$.  It
follows the proof of Lemma 2.2 in \cite{duke}, that
\[ \left\|T_0 - T_\delta\right\|_{S^p(L^2(\sphere))} = \left( \sum_{n \geq
  0} (2n+1) |P_n(0)-P_n(\delta)|^p\right)^{1/p}.\] Here $2n+1$ appears as
the dimension of the space $H_n$ of restrictions to $\sphere$ of the
harmonic homogeneous polynomials of degree $n$ on $\R^3$ (more precisely
$L^2(\sphere)$ decomposes as $\oplus_{n\geq 0} H_n$, and $T_\delta$ acts
as the multiplication by $P_n(\delta)$ on $H_n$).

If $|\delta|\leq 1/2$, the estimate $|P_n(0)-P_n(\delta)| \leq C \min(n
|\delta|,1)/\sqrt{n+1}$ for some constant $C$ was proved in the proof of
Lemma 2.2 in \cite{duke} and implies the first inequality of Lemma
\ref{lem-Tkk-1_R}.

The second inequality holds because the function $1$ on $\sphere$ is an
eigenvector with eigenvalue $1$ for all the $T_\delta$'s.
\end{proof}

For any $s,t \in \R_+$ (the non-negative real numbers) , we denote
\[D(s,t) = e^{-\frac{s+2t}{3}}\begin{pmatrix} e^{s+t} & 0&0\\0& e^t & 0 \\
  0&0&1\end{pmatrix}.\]

\begin{lemma}\label{lemma=key_estimate_SL3R}Let $\varphi \in C(G)$ be a $K$-biinvariant function,
  $s,t,s',t' \in \R_+$, and $C_1$ the constant in Lemma \ref{lem-Tkk-1_R}.
\begin{itemize}
\item If $s+2t=s'+2t'$ and $0 \leq t' \leq t \leq s+t\leq s'+t' \leq s+2t$,
then
\[|\varphi(D(s,t)) - \varphi(D(s',t'))| \leq C_1 e^{-(1/2-2/p) t'} \|\chap \varphi\|_{MS^p(L^2(G))}.\]
\item If $2s+t=2s'+t'$ and $0 \leq s' \leq s \leq s+t\leq s'+t' \leq 2s+t$,
then
\[|\varphi(D(s,t)) - \varphi(D(s',t'))| \leq C_1 e^{-(1/2-2/p) s'} \|\chap \varphi\|_{MS^p(L^2(G))}.\]
\end{itemize}
\end{lemma}
\begin{proof}[Sketch of proof] As in the proof of Lemma \ref{lemTrenf}, the
  second inequality follows from the first by inversing the role of $s,s'$
  and $t,t'$.

  Let us now fix $s,t,s',t'$ as in the first inequality. We can assume that
  $e^{-t'} \leq 1/2$ because otherwise the inequality $\|\varphi\|_\infty
  \leq \|\chap \varphi\|_{MS^p(L^2(G))}$ implies that the desired
  inequality holds with $C_1=2$. In \cite{duke} the first author
  constructed two continuous injective maps $\alpha,\beta:\sphere \to G/K$
  such that there is some for some $0 \leq \delta \leq e^{-t'}$
  satisfying~:
\begin{gather}
\label{eq=alpha_beta_orth_R}\alpha(x)^{-1} \beta(y) = K D(s,t) K \textrm{
  if } \langle x,y \rangle =0\\
\label{eq=alpha_beta_Nonorth_R} \alpha(x)^{-1} \beta(y) = K D(s',t') K
\textrm{ if }\langle x,y \rangle =\delta
\end{gather}
This is contained in Lemma 2.7 in \cite{duke}, with $\alpha(\cdot) =
q_{-(s+t)}(\cdot)$ and $\beta(\cdot)= q_t(\cdot)$.

Let $\mu$ be some Radon measure on $G/K$ with full support such that the
image measures of the measure $dx$ on $S^2$ by $\alpha$ and $\beta$ are
absolutely continuous with respect to $\mu$. By Theorem
\ref{thm=mult_de_symb_continu} we have that
\[\|\chap \varphi\|_{cbMS^p(L^2(G/K,\mu))} \leq \|\chap
\varphi\|_{cbMS^p(L^2(G))} \leq \|\chap \varphi\|_{MS^p(L^2(G))}.\] 
The image measures of $dx$ by $\alpha$ and $\beta$ are absolutely
continuous with respect to $\mu$, and since $\sphere$ is compact $\alpha$
and $\beta$ are homeomorphisms onto their images. Therefore, as in Lemma
\ref{lemma=change_of_measure}, $\alpha$ and $\beta$ induce isometries
$U_\alpha,U_\beta:L^2(\sphere) \to L^2(G/K,\mu)$ and hence an isometric
embedding $i:S^p(L^2(\sphere)) \to S^p(L^2(G/K),\mu)$ given by
$i(T)=U_\alpha T U_\beta^*$. It is straightforward to see that
\eqref{eq=alpha_beta_orth_R} (resp. \eqref{eq=alpha_beta_Nonorth_R})
implies $M_{\chap \varphi}(i(T_0))=\varphi(D(s,t)) i(T_0)$ (resp. $M_{\chap
  \varphi}(i(T_\delta))=\varphi(D(s',t')) i(T_\delta)$). We thus get
\[ \|\varphi(D(s,t)) T_0 - \varphi(D(s',t')) T_\delta\|_p \leq \|\chap
\varphi\|_{MS^p(L^2(G))} \|T_0 - T_\delta\|_p.\]
Lemma \ref{lem-Tkk-1_R} and the inequality $ |\delta| \leq e^{-t'}$ allows
to conclude the proof.
\end{proof}

\begin{proof}[Proof of Lemma \ref{lemma=SL3R_main_estimate}.]
We copy the proof of \cite{duke}, Proposition 2.3. Take $\varphi \in
C_0(G)$. Assume for simplicity $\|\chap \varphi\|_{MS^p(L^2(G))}=1$. Let
$u,v \in \R_+$ such that $u/v \in ]1,2[$. Apply the first part of Lemma
\ref{lemma=key_estimate_SL3R} to $(s,t)=(2v-u,2u-v)$ and $(s',t') = (u,u)$
and get
\[ |\varphi(D(u,u)) - \varphi(D(2v-u,2u-v))| \leq C_1 e^{-(1/2-2/p) u}.\]
Apply the second part of Lemma
\ref{lemma=key_estimate_SL3R} to $(s,t)=(v,v)$ and $(s',t') = (2v-u,2u-v) $
and get
\[ |\varphi(D(v,v)) - \varphi(D(2v-u,2u-v))| \leq C_1 e^{-(1/2-2/p) (2v
  -u)}.\]
Hence,
\[|\varphi(D(v,v)) - \varphi(D(u,u))| \leq C_1 \left( e^{-(1/2-2/p) u} + e^{-(1/2-2/p) (2v
  -u)}\right).\] Taking $u/v$ close enough to $1$, we can have
$(1/2-2/p)(2v-u) \geq \varepsilon v$ and we deduce easily Lemma
\ref{lemma=SL3R_main_estimate}.
\end{proof}

We are now able to prove the main results of the introduction in the
real case.

\begin{proof}[Proof of Theorem \ref{thm=main_thm_continu_R}.]
This is immediate from Proposition \ref{prop=SLrR}.
\end{proof}

\begin{proof}[Proof of Theorem \ref{thm=main_theorem_R}.]
Theorem \ref{thm=main_thm_continu_R} and Theorem
\ref{thm=characterzation} imply that $\Gamma$ does not have
$\SCHURCBAP{p}$. We conclude using Corollary
\ref{coro=OAP_Lp_implieque_notrePte}.
\end{proof}

\begin{proof}[Proof of Theorem \ref{thm=main_theorem_AP} (real case).]
If $4<p<\infty$, as in the proof above, $\Gamma$ does not have
$\SCHURCBAP{p}$. The theorem thus from Corollary
\ref{coro=AP_implique_notrePte}.
\end{proof}

\end{document}